\documentclass{article}
\usepackage{amsfonts}
\usepackage{mathrsfs}
\usepackage{amssymb, latexsym}
\usepackage{graphicx}
\usepackage{amsmath}
\usepackage[pdftex]{color} 

\definecolor{flesh}{rgb}{.90,.60,.50}
\definecolor{night}{rgb}{.05,0,.05}
\definecolor{gold}{rgb}{.80,.75,.10}

\definecolor{wine}{rgb}{1.00,0,.10}
\definecolor{dimmed}{gray}{0.9}

\definecolor{blau1}{rgb}{0,0.3,0.65}
\definecolor{red1}{rgb}{.5,.0,.0}
\definecolor{green1}{rgb}{.0,.5,.0}
\definecolor{water}{rgb}{.0,.9,0.7}
\definecolor{panelbackground}{rgb}{0,0.39,0.61}
\definecolor{blue1}{rgb}{0.367,0.679,0.839}
\definecolor{blue2}{rgb}{0.000,0.391,0.605}
\definecolor{blue3}{rgb}{0.4,0.8,0.9}

\def\en t{{{\rm Z}\mkern-5.5mu{\rm Z}}}

\newtheorem{theorem}{Theorem}[section]

\newtheorem{corollary}[theorem]{Corollary}

\newtheorem{lemma}[theorem]{Lemma}

\newtheorem{proposition}[theorem]{Proposition}

\newtheorem{remark}[theorem]{Remark}

\newenvironment{proof}[1][Proof]{\textbf{#1.} }{\ \rule{0.5em}{0.5em}}

\begin{document}

\title{The Navier-Stokes-$\alpha$ equation via forward-backward stochastic differential systems \footnote{Address:  
Academy of Mathematics and Systems Science, Chinese Academy of Sciences, No. 55, 
Zhongguancun East Road, Beijing, 100190, China; Grupo de F\'isica Matem\'atica da Universidade de Lisboa and Dep. de Matem\'atica Instituto Superior T\'ecnico (Universidade de Lisboa), Av. Rovisco Pais, 1049-001 Lisboa, Portugal. E-mail: liuguoping@amss.ac.cn.
}}

\author{Guoping Liu}


%
%

\maketitle
\tableofcontents

%
%

\section{Introduction}\label{section1}
\setcounter{equation}0

The two-dimensional Navier-Stokes-$\alpha$ equation (c.f. \cite{Foias-Holm-Titi:01}) is given by
\begin{eqnarray}\label{eq1.1}
&&\partial_t(u-\alpha^2\Delta u)-\nu \Delta(u-\alpha^2\Delta u)+u\cdot \nabla(u-\alpha^2\Delta u)=\alpha^2\nabla u^T\Delta u-\nabla p, \\\nonumber
&&\nabla\cdot u=0,
\end{eqnarray}
where $u=(u^1, u^2)$ is the velocity field, $\nu$ the viscosity constant and $p$ the pressure.

This equation was studied by various authors. Without pretending to be exhaustive, let us mention
\cite{Bjorland-Schonbek:08}(Theorem 4.2), where Bjorland and Schonbek proved the existence and uniqueness of weak solutions of (n=2,3,4) dimensional viscous Camass-Holm equation $(\alpha=1)$
on  open bounded sets  or  in $\mathbb R^n$ in various Sobolev spaces.
In  \cite{Ilyin-Titi:03} (pages 754-756),  Ilyin and Titi proved that if $\phi_0 \in H=L^2 \bigcap \{\int \phi dx=0\}$, the 2D viscous vorticity  Camass-Holm equation 
on the torus
$$\partial_t \phi-\nu\Delta \phi+u. \nabla \phi=0, \phi=\omega-\Delta \omega, \omega= \operatorname{rot} u$$
has a unique solution $\phi \in C([0,T];H)\bigcap L^2([0,T]; H^1)$.

Existence and uniqueness of solutions for the viscous Camassa-Holm equation on periodic domains in three dimensions was proved  in \cite{Foias-Holm-Titi:02} using the Galerkin method (Theorem 3); a general existence and uniqueness theorem in three dimensions is provided in \cite{Marsden-Shkoller:01} using a fixed point argument (Theorem 5.2).

In this paper we study Navies-Stokes-alpha equations using probabilistic methods. More precisely we use the representation of the p.d.e. solutions by forward-backward stochastic differential equations, which was presented in \cite{Cruzeiro-Shamarova:09} for the Navier-Stokes case. The forward-backward stochastic systems in question are infinite dimensional and, in order to solve them, one needs to have good estimates of the operators involved. These estimates depend on the underlying spaces, dimensions, etc, so that one cannot apply a "general theory", but instead have to work carefully each case.

Here we consider two different cases, namely  the periodic $2$-dimensional Navier-Stokes-alpha equation using the corresponding vorticity model,
as in \cite{Cruzeiro-Qian:14}, and the  Navier-Stokes-alpha equation in the whole space $\mathbb R^d$. 

For the periodic $2$-dimensional case, we first derive our problem in Section \ref{section2}. After that we show the existence of bounded solutions using the associated forward-backward stochastic differential equation in Section \ref{section3}.
For the d dimensional $(d\geq 3)$ situation our methods were inspired by \cite{Chen-Cruzeiro-Qian:13}.
Letting $m=u-\alpha^2\Delta u$, in section \ref{section4}, we study the local existence and uniqueness of the solution to d-dimensional $(d\geq 3)$ Navier-Stokes-alpha equation 
\begin{eqnarray}\label{eq1.2}
&&\partial_t m-\nu \Delta m +u\cdot\nabla m=-\nabla p+\alpha^2\nabla u^T\Delta u, \\\nonumber
&&\nabla\cdot m(t)=0,~~~m(0)=m_0,~~~t\in[0,T],
\end{eqnarray}
where $u:[0,T]\times\mathbb{R}^d\rightarrow \mathbb{R}^d$ is the velocity field,
$\nu>0$ the viscosity constant and 
$p:[0,T]\times\mathbb{R}^d\rightarrow \mathbb{R}$ the pressure. From the above equation, we can deduce that $p$ satisfies the following equation:
\begin{eqnarray}\label{eq1.3}
\Delta p&=&-\sum_{i,j=1}^d [\partial_i u^j\partial_j (u^i-\alpha^2\Delta u^i)-\alpha^2\partial_{ii}^2u^j\Delta u^j-\alpha^2\partial_iu^j\Delta\partial_iu^j]\\\nonumber
&=&-\sum_{i=1}^d\partial_i\sum_{j=1}^d [u^j\partial_j(u^i-\alpha^2\Delta u^i)-\alpha^2\partial_iu^j\Delta u^j],~~~\forall t\in [0,T],
\end{eqnarray}
where $\partial_i u^j, 1\leq i,j\leq d$ denotes the partial derivative with respect to the 
$i$-th variable for the $j$-th component of $u$. Both in the proof of the 2-dimensional and d-dimensional case, we have used the fixed point theorem.

\section{Formulation of the two-dimensional problem}\label{section2}
\setcounter{equation}0

In this section we derive the backward stochastic differential equation associated with the two-dimensional Navier-Stokes-$\alpha$ equation (\ref{eq1.1}).

Differentiating equation (\ref{eq1.1}), we get
\begin{eqnarray*}
&&\partial_t(\partial_1u^2 -\alpha^2\Delta\partial_1 u^2)+u^k\partial_k(\partial_1u^2-\alpha^2\Delta \partial_1u^2)
+ \partial_1 u^k\partial_k(u^2-\alpha^2\Delta u^2)\\
 &=&\nu\Delta(\partial_1u^2-\alpha^2\Delta\partial_1 u^2)+\alpha^2\partial_1\partial_2u^k\Delta u^k+\alpha^2\partial_2u^k\Delta\partial_1 u^k-\partial_1\partial_2 p.
\end{eqnarray*}
Therefore, we have
\begin{eqnarray*}
&&\partial_t[(\partial_1u^2-\partial_2u^1)-\alpha^2\Delta(\partial_1 u^2-\partial_2u^1)]
+u^k\partial_k[(\partial_1u^2-\partial_2u^1)-\alpha^2\Delta(\partial_1 u^2-\partial_2u^1)]\\
&&+ [\partial_1 u^k\partial_k(u^2-\alpha^2\Delta u^2)-\partial_2 u^k\partial_k(u^1-\alpha^2\Delta u^1)]\\
&=&\nu\Delta[(\partial_1u^2-\partial_2u^1)-\alpha^2\Delta(\partial_1 u^2-\partial_2u^1)]
+\alpha^2[\partial_2u^k\Delta\partial_1 u^k-\partial_1u^k\Delta\partial_2 u^k].
\end{eqnarray*}
Since $\nabla\cdot u=0$, by simple calculation, we have
$$-[\partial_1u^k\partial_k(u^2-\alpha^2\Delta u^2)-\partial_2u^k\partial_k(u^1-\alpha^2\Delta u^1)]
+\alpha^2[\partial_2u^k\Delta \partial_1u^k-\partial_1u^k\Delta \partial_2u^k]=0.$$
We denote by
\begin{eqnarray}\label{eq2.1}
\omega=\frac{\partial u^2}{\partial x^1}-\frac{\partial u^1}{\partial x^2}
\end{eqnarray}
the vorticity of $u$, then $\omega$ satisfies the evolution equation 
\begin{eqnarray}\label{eq2.2}
\partial_t(\omega-\alpha^2\Delta \omega)-\nu\Delta (\omega-\alpha^2\Delta \omega)+u\cdot\nabla (\omega-\alpha^2\Delta \omega)=0. 
\end{eqnarray}
Equation (\ref{eq2.2}) is the vorticity equation which is equivalent to the Navier-Stokes-$\alpha$ equation. 

Let $B=(B^1,B^2)$ be a standard Brownian motion on a complete probability space $(\Omega,\mathcal{F},P)$, and define
the random processes (we suppress the probability parameter in the notations):
\begin{eqnarray*}
&&Y(t,x)=(\omega-\alpha^2\Delta\omega)(T-t,x+\sqrt{2\nu}B_t)\\
&&Z(t,x)=\nabla(\omega-\alpha^2\Delta\omega)(T-t,x+\sqrt{2\nu}B_t)
\end{eqnarray*}
for $(t,x)\in [0,T]\times\mathbb{R}^2$.
Suppose that $u(0,x)=\varphi(x)$ is a smooth vector field with period one, that is, $\varphi(x+e_i)=\varphi(x)$ for all $x\in \mathbb{R}^2$, where $e_1=(1,0)$ and $e_2=(0,1)$ are the standard basis in $\mathbb{R}^2$. Then the unique solution $(u,p)$ of equation (\ref{eq1.1}) is smooth and periodic in space variables (with period one).
Let $\psi=(\frac{\partial\varphi^2}{\partial x^1}-\frac{\partial\varphi^1}{\partial x^2})-\alpha^2\Delta(\frac{\partial\varphi^2}{\partial x^1}-\frac{\partial\varphi^1}{\partial x^2})$ and $\xi(x)=\psi(x+\sqrt{2\nu}B_T)$.
If $\psi$ is smooth, the random variable $\xi$ is also smooth and periodic in $x$. By It\^o's formula
\begin{eqnarray}\label{eq2.3}
\xi(x)-Y(t,x)&=&\sqrt{2\nu}\int_t^T<\nabla(\omega-\alpha^2\Delta\omega)(T-s,x+\sqrt{2\nu}B_s),dB_s>\\ \nonumber
&&+\int_t^T[-\frac{\partial}{\partial s}(\omega-\alpha^2\Delta\omega)+\nu\Delta(\omega-\alpha^2\Delta\omega)]
(T-s,x+\sqrt{2\nu}B_s)ds.
\end{eqnarray}
Using the  vorticity equation (\ref{eq2.2}), we  obtain
\begin{eqnarray}\label{eq2.4}
\xi(x)-Y(t,x)=\sqrt{2\nu}\int_t^T <Z(s,x),dB_s>+\int_t^T<Z(s,x),X(s,x)>ds
\end{eqnarray}
where
$$X(t,x):=u(T-t,x+\sqrt{2\nu}B_t)$$
which is continuous in t, smooth  and periodic in x.  In order to write $X(t,x)$ in terms of 
$Y$ and $Z$ we proceed as follows.

Due to the divergence-free condition, the relationship between the vorticity $\omega$
and the associated vector field $u$ is determined by the Poisson equations 
\begin{eqnarray}\label{eq2.5}
\begin{cases}
    $$\Delta u^1=-\frac{\partial\omega}{\partial x^2}$$&\mbox{}\\
    $$\Delta u^2=\frac{\partial\omega}{\partial x^1}.$$ &\mbox{}
\end{cases} 
\end{eqnarray}
We consider the linear operators $K_i: \omega\rightarrow u^i$ (where  $i=1,2$)
and $K=(K_1,K_2)$ defined by solving the Poisson equations (\ref{eq2.5}), where $\omega$ is a real function with period one and mean zero (i.e. $\int_{[0,1)^2} \omega(t,x)dx=0,~\forall t\geq 0$).

Denote by $\mathbb{T}^2$  the two-dimensional torus equipped with the standard metric and the Lebesque measure. We 
 identify tensor fields in $\mathbb{R}^2$ with period one  with the corresponding tensor fields on $\mathbb{T}^2$ in the canonical way.
In particular we have
$$L^2(\mathbb{T}^2)=\{f\in L_{loc}^2(\mathbb{R}^2): f(\cdot+e_i)=f(\cdot) ~~for ~~i=1,2\}\cap L^2([0,1)^2).$$
If $f\in L^2(\mathbb{T}^2)$ then 
\begin{eqnarray}\label{eq2.6}
f(x)=\sum_{k\in \mathbb{Z}^2}e^{2\pi i<k,x>}\hat{f}(k)
\end{eqnarray}
where $ \hat{f}(k)=\int_{[0,1)^2}e^{-2\pi i<k,y>}f(y)dy, ~~k\in\mathbb{Z}^2$
is the Fourier transform of $f$.
We recall Green's formula
for the Poisson equation 
\begin{eqnarray}\label{eq2.7}
\Delta g=-f ~~in ~~\mathbb{T}^2,~~~~~~\int_{\mathbb{T}^2}g(y)dy=0,
\end{eqnarray}
where $\int_{\mathbb{T}^2}f(y)dy=0$ and $f\in L^2(\mathbb{T}^2)$. The unique solution of problem (\ref{eq2.7}) is given by 
\begin{eqnarray}\label{eq2.8}
g(x)=\sum_{k\in\mathbb{Z}^2,k\neq 0}\frac{e^{2\pi i<k,x>}}{4\pi|k|^2}\hat{f}(k).
\end{eqnarray}
Applying Green's formula to the vorticity $\omega$, we have 
\begin{eqnarray}
u (t,x)=(K^1 , K^2 ) (\omega (t,\cdot ))(x)= \frac{i}{2} \sum_{k=(k_1,k_2)\in \mathbb{Z}^2, k\neq 0}\frac{1}{|k|^2}(k_2,-k_1)e^{2\pi i<k,x>}
\widehat{\omega(t,\cdot)}(k).
\end{eqnarray}
On the other hand
\begin{eqnarray*}
\widehat{(I-\alpha^2\Delta)^{-1}Y}(t,k)&=&\int_{[0,1)^2}e^{-2\pi i<k,y>}\omega(T-t, y+\sqrt{2\nu}B_t)dy\\
&=&e^{2\pi i<k,\sqrt{2\nu}B_t>}\int_{[0,1)^2+\sqrt{2\nu}B_t}e^{-2\pi i<k,y>}\omega(T-t, y)dy
\\
&=&e^{2\pi i<k,\sqrt{2\nu}B_t>}\int_{[0,1)^2}e^{-2\pi i<k,y>}\omega(T-t, y)dy\\
&=&e^{2\pi i<k,\sqrt{2\nu}B_t>}\widehat{\omega(\cdot)}(T-t,k).
\end{eqnarray*}
Therefore, we obtain 
\begin{eqnarray}\label{eq2.10}
X(t,x)=\frac{i}{2}\sum_{k=(k_1, k_2)\in \mathbb{Z}^2, k\neq 0}\frac{1} {|k|^2}(k_2 ,-k_1 )e^{2\pi i<k,x>}
\widehat{(I-\alpha^2\Delta)^{-1}Y}(t,k)
\end{eqnarray}
and we have
\begin{eqnarray}\label{eq2.11}
X^j(t,x)=K_j((I-\alpha^2\Delta)^{-1}Y(t,\cdot))(x)\triangleq \widetilde{K}_j^\alpha(Y(t,\cdot))(x) ~~~~~\forall x\in \mathbb{R}^2, j=1,2.
\end{eqnarray}
From (\ref{eq2.4}), we conclude that 
\begin{eqnarray}\label{eq2.12}
\xi(x)-Y(t,x)=\sqrt{2\nu}\int_t^T<Z(s,x),dB_s>+\int_t^T<Z(s,x),\widetilde{K}^\alpha(Y(s,\cdot))(x)>ds
\end{eqnarray}
with $x\in \mathbb{R}^2$.

We have, as in \cite{Cruzeiro-Qian:14}, transformed the initial-valued vorticity equation (\ref{eq2.2}) in a stochastic backward system (\ref{eq2.12}) with terminal value,
that can be written in its differential form as:
\begin{eqnarray}\label{eq2.13}
dY=<Z,\tilde{K}^\alpha(Y)>dt+\sqrt{2\nu}<Z,dB>,~~~~~Y_T=\xi.
\end{eqnarray}

\section{Existence and uniqueness results on  two-dimensional torus}\label{section3}
\setcounter{equation}0

Let $\mathcal{F}_t^0=\sigma\{B_s:s\leq t\}$ and $(\mathcal{F}_t)_{t\geq 0}$ be the completed continuous filtration associated with $(\mathcal{F}_t^0)_{t\geq 0}.$
We assume that the terminal value $\xi$ is a bounded random function on $\Omega \times \mathbb{T}^2$, $ess\sup\limits_{ \Omega \times \mathbb{T}^2} |\xi | \leq C_1$, which is 
$\mathcal{F}_T\otimes \mathcal{B}(\mathbb{T}^2)$ measurable and that $\int_{\mathbb{T}^2} \xi (y)dy=0$ a.e..

We state our main result.
\begin{theorem}\label{thm 3.1}
Under the above assumptions on the terminal value $\xi$, there exists a unique solution  $(Y, Z)$ of the BSDE (\ref{eq2.13}) such that
 $ess\sup\limits_{[0,T]\times \Omega} \| Y_t  \| < +\infty $ and,
for almost all $x\in \mathbb{T}^2$, the  It\^o's integral $\int_0^{\cdot} <Z_s (x),dB_s >$ is a BMO martingale, and 
\begin{eqnarray}\label{eq3.1}
ess\sup\limits_{[0,T]\times \Omega}\mathbb{E}\{\int_t^T \|Z_s\|^2|\mathcal{F}_t\}<+\infty
\end{eqnarray}
where $\|\cdot\|$ denotes the $L^2$-norm on $\mathbb{T}^2$.

Furthermore, if the random variable $\xi$ is of the form $\xi(x) =\psi (x +\sqrt{2\nu} B_T )$, then 
$Y(t,x)=(\omega-\alpha^2\Delta\omega)(T-t,x+\sqrt{2\nu}B_t)$
with $\omega$ solving the vorticity equation (\ref{eq2.2}).
\end{theorem}

In the rest of this section we prove Theorem \ref{thm 3.1},  following the lines of \cite{Cruzeiro-Qian:14} adapted to our different operators.

\subsection{Preliminaries}
If  $k\in\mathbb{Z}_{+}$ and $q\geq 1$, we consider the Sobolev space 
$$W^{k,q}(\mathbb{T}^2)=\{f: \partial^\alpha f\in L_{loc}^q(\mathbb{R}^2)\cap L^q([0,1)^2)  ~~for~~|\alpha|\leq k 
 \\~~and~~f(\cdot+e_i)=f(\cdot)  ~~for~~i=1,2 \}$$
together with the Sobolev norm 
$$\|f\|_{k,q}=(\sum_{\alpha\in\mathbb{Z}^2,|\alpha|\leq k}\|\partial^\alpha f\|_q^q)^{\frac{1}{q}}$$
where $\|\cdot\|_q$ is the $L^q$-norm over $\mathbb{T}^2$.
If $q=2$ then we use $\|\cdot\|$ instead of $\|\cdot\|_2$ for simplicity.

If $f\in L^2(\mathbb{T}^2)$ such that $\int_{[0,1)^2}f=0$, then $g_j=K_j(f)$ (where the operator $K$ was defined in the previous section) are the unique solutions with period one of the Poisson equations 
\begin{eqnarray}\label{eq3.2}
\Delta g_1=-\frac{\partial f}{\partial x_2},~~~~\Delta g_2=\frac{\partial f}{\partial x_1}~~~~ 
on~~ \mathbb{T}^2.
\end{eqnarray}
such that $\int_{[0,1)^2}g_j=0$.
Based on this, we have
$$\int_{\mathbb{T}^2}|\nabla g_j|^2=-\int_{\mathbb{T}^2}g_j\Delta g_j
=\int_{\mathbb{T}^2} g_1\frac{\partial f}{\partial x_2}~~~~or~~~~-\int_{\mathbb{T}^2} g_2\frac{\partial f}{\partial x_1}$$
according to $j=1$ or $j=2$. Integration by parts together with Cauchy-Schwartz's inequality applied to the last integrals imply that
$$\int_{\mathbb{T}^2}|\nabla g_j|^2\leq \sqrt{\int_{\mathbb{T}^2}|\nabla g_j|^2} \sqrt{\int_{\mathbb{T}^2}|f|^2}$$
which yields 
\begin{eqnarray}\label{eq3.3}
\|\nabla K_j(f)\|\leq \|f\|, ~~~~j=1,2.
\end{eqnarray}
Let $\lambda_1>0$ be the spectral gap for the torus $\mathbb{T}^2$. Since 
$\int_{\mathbb{T}^2}K_j(f)=0$, according to the Poincar\'e inequality 
\begin{eqnarray}\label{eq3.4}
&&\|K_j(f)\|\leq \frac{1}{\sqrt{\lambda_1}}\|\nabla K_j(f)\|\leq \frac{1}{\sqrt{\lambda_1}}\|f\|,\\\nonumber
&&\|\widetilde{K}_j^\alpha(f)\|=\|K_j((I-\alpha^2\Delta)^{-1}f)\|\leq  \frac{1}{\sqrt{\lambda_1}}\|(I-\alpha^2\Delta)^{-1}f\|\leq C(\alpha)\|f\|_{-2,2}.
\end{eqnarray}
Therefore  there exists  a  constant $C_0(\alpha)>0$ (c.f. \cite{Agmon-Douglis:59,Agmon-Douglis:64,Morrey-66}) such that 
\begin{eqnarray}\label{eq3.5}
\|\widetilde{K}_j^\alpha(f)\|_{k,2} \leq C_0(\alpha)\|f\|_{k-3,2}
\end{eqnarray}
for  every $f\in W^{k-3,2}(\mathbb{T}^2)$ with $\int_{\mathbb{T}^2}f=0$, $k\in \mathbb{N}$.

Let us consider the following linear BSDE
\begin{eqnarray}\label{eq3.6}
&&dY(t,x)=<Z(t,x),h(t,x)>dt+<Z(t,x),dB_t>,\nonumber\\
&&Y(T,x)=\xi(x),
\end{eqnarray}
where $h\in \mathcal {O}\times \mathcal{B}(\mathbb{R}^2)$ is a given $\mathbb{T}^2$-valued  optional process such that, for each $(w,t)\in \Omega\times[0,T]$, $h(w,t,\cdot)\in C(\mathbb{T}^2)$ and
\begin{eqnarray}\label{eq3.7}
\mathbb{E}\int_0^T |h(t,x)|^2dt<\infty~~\forall x\in \mathbb{T}^2.
\end{eqnarray}
Let $ \xi \in L^\infty(\Omega\times \mathbb{T}^2)$ which is $\mathcal{F}_T$-measurable.

The linear equation (\ref{eq3.6}) may be solved for every $x\in \mathbb{T}^2$. More precisely, for each $x\in\mathbb{T}^2$, since we have (\ref{eq3.7}), we can define a probability $\mathbb{Q}^x$ on $\mathcal{F}_T$ by $\frac{d\mathbb{Q}^x}{d\mathbb{P}}=R(T,x)$, where  
$$R(t,x)=\exp[-\int_0^t<h(s,x),dB_s>-\frac{1}{2}\int_0^t|h(s,x)|^2ds].$$
If $(Y(\cdot,x),Z(\cdot,x))$ is the unique solution of (\ref{eq3.6}),  the Girsanov theorem shows that $Y(\cdot,x)$ must be a martingale under the new probability $\mathbb{Q}^x$; hence,
$$Y(t,x)=\mathbb{E}^{\mathbb{Q}^x}\{\xi(x)|\mathcal{F}_t\}$$
which implies that 
$$Y(t,x)=\mathbb{E}\{\frac{R(T,x)}{R(t,x)}\xi(x)|\mathcal{F}_t\}$$
for $(t,x)\in[0,T]\times\mathbb{T}^2$. 
Therefore, since $\xi$ is a bounded  $\mathcal{F}_T$-measurable random variable, and assuming that 
$h$ is a $C(\mathbb{T}^2)$-valued adapted stochastic process satisfying (\ref{eq3.7}), 
the unique solution to (\ref{eq3.6}) is given by 
\begin{eqnarray}\label{eq3.8}
Y(t,x)=\mathbb{E}\{\xi(x)e^{-\int_t^T<h(s,x),dB_s>-\frac{1}{2}\int_t^T|h(s,x)|^2ds}|\mathcal{F}_t\}
\end{eqnarray}
with $(t,x)\in[0,T]\times \mathbb{T}^2$.

\subsection{Proof of Theorem \ref{thm 3.1}}
Let $\mathcal{H}$ denote the set of bounded $\mathcal{P}\times \mathcal{B}(\mathbb{R}^2)$-predictable stochastic processes 
Y on $\Omega\times[0,T]\times \mathbb{T}^2$ that satisfy the following conditions:

(i) For every $x\in \mathbb{T}^2$, $Y(\cdot, x)$ is a continuous semimartingale (up to time $T$) on $(\Omega, \mathcal{F},\mathcal{F}_t, \mathbb{P})$ 
and $Y_T=\xi$;

(ii) If $M$ is the  martingale part  of $Y$,
$M(t,x)=\int_0^t<Z(t,x),dB_t>$
where $Z$ is $\mathcal{P}\times \mathcal{B}(\mathbb{R}^2)$-measurable, then 
$$ess\sup\limits_{[0,T]\times \Omega}\mathbb{E}\{\int_t^T \|Z_s\|^2|\mathcal{F}_t\}<+\infty.$$

Let $Y^\alpha\in\mathcal{H}$ and  define $\tilde{Y}^\alpha=\mathcal{L}(Y^\alpha)$ by solving the following linear BSDE
\begin{eqnarray}\label{eq3.9}
&&d\tilde{Y}^\alpha(t,x)=<\tilde{Z}^\alpha(t,x),\tilde{K}^\alpha(Y^\alpha(t,\cdot))(x)>dt+\sqrt{2\nu}<\tilde{Z}^\alpha(t,x),dB_t>,\\\nonumber
&&\tilde{Y}^\alpha(T,x)=\xi(x),
\end{eqnarray}
for every $x\in \mathbb{T}^2$. Then $\tilde{Y}^\alpha\in\mathcal{H}.$
Suppose $Y^\alpha\in\mathcal{H}$ is such that $\|Y^\alpha\|_\infty\leq C_1$, where
$\|Y^\alpha\|_\infty $ is the essential bound of $Y^\alpha$ on $\Omega \times [0,T]\times \mathbb T^2$. Define $\tilde{Y}^\alpha=\mathcal{L}(Y^\alpha)$ and $\tilde{Z}^\alpha$
the density process of the martingale part of $\tilde{Y}^\alpha$, that is,  $(\tilde{Y}^\alpha,\tilde{Z}^\alpha)$ is the solution of the linear BSDE (\ref{eq3.9}),
where $|\xi(w,t,x)|\leq C_1$. By the maximal principle stated in \cite{Cruzeiro-Qian:14} (Lemma 4.2), we have  $|\tilde{Y}^\alpha(w,t,x)|\leq C_1$.

According to It\^o's formula,
$$|\tilde{Y}_t^\alpha|^2=|\xi|^2-2\nu\int_t^T|\tilde{Z}^\alpha|^2ds -2\int_t^T \tilde{Y}^\alpha<\tilde{Z}^\alpha,\tilde{K}^\alpha(Y^\alpha)>ds-2\sqrt{2\nu}\int_t^T\tilde{Y}^\alpha<\tilde{Z}^\alpha,dB>.$$
Taking conditional expectations, we obtain
\begin{eqnarray*}
 |\tilde{Y}_t^\alpha|^2+2\nu\mathbb{E}^{\mathcal{F}_t}\int_t^T|\tilde{Z}^\alpha|^2ds&=&\mathbb{E}^{\mathcal{F}_t}|\xi|^2 -2\mathbb{E}^{\mathcal{F}_t}\int_t^T \tilde{Y}^\alpha<\tilde{Z}^\alpha,\tilde{K}^\alpha(Y^\alpha)>ds\\
 &\leq&C_1^2+2C_1\mathbb{E}^{\mathcal{F}_t}\int_t^T |<\tilde{Z}^\alpha,\tilde{K}^\alpha(Y^\alpha)>|ds.
\end{eqnarray*}
Integrating over $\mathbb{T}^2$, using Young's inequality and estimate (\ref{eq3.4}), we have 
\begin{eqnarray*}
\|\tilde{Y}_t^\alpha\|^2+2\nu\mathbb{E}^{\mathcal{F}_t}\int_t^T\|\tilde{Z}^\alpha\|^2ds
&\leq&C_1^2+2C_1\mathbb{E}^{\mathcal{F}_t}\int_t^T |<\tilde{Z}^\alpha,\tilde{K}^\alpha(Y^\alpha)>|ds\\
&\leq&C_1^2+2C_1\mathbb{E}^{\mathcal{F}_t}\int_t^T \|\tilde{Z}^\alpha\|\|\tilde{K}^\alpha(Y^\alpha)\|ds\\
&\leq&C_1^2+C_1\mathbb{E}^{\mathcal{F}_t}\int_t^T [\frac{1}{\varepsilon}\|\tilde{Z}^\alpha\|^2+\varepsilon\|\tilde{K}^\alpha(Y^\alpha)\|^2]ds\\
&\leq&C_1^2+\frac{C_1}{\varepsilon}\mathbb{E}^{\mathcal{F}_t}\int_t^T \|\tilde{Z}^\alpha\|^2 ds
+\varepsilon C_1C^2(\alpha)\mathbb{E}^{\mathcal{F}_t}\int_t^T \|Y^\alpha\|_{-2,2}^2 ds
\end{eqnarray*}
for every $\varepsilon>0$. Recall that 
$$\|\tilde{Z}^\alpha\|_{BMO}^2=ess\sup\limits_{[0,T]\times \Omega}\mathbb{E}^{\mathcal{F}_t}\int_t^T \|\tilde{Z}^\alpha\|^2 ds$$
and
$$\|Y^\alpha\|_{-2,2}^2\triangleq\int_{\mathbb{R}^2}(1+|\xi|^2)^{-2}|\widehat{Y^\alpha}(\xi)|^2d\xi,$$
where $\widehat{Y^\alpha}(\xi)=\int_{\mathbb{T}^2}e^{-2\pi i<\xi,x>}Y^\alpha(t,x)dx$.
By the definition of $\|Y^\alpha\|_{-2,2}$, we know that:
if $\|Y^\alpha\|_{\infty}\leq C_1$, then $\|Y^\alpha\|_{-2,2}\leq C_1$.
Consequently, we have
\begin{eqnarray}\label{eq3.10}
\|\tilde{Z}^\alpha\|_{BMO}^2\leq\frac{C_1}{2\nu}[C_1+T\varepsilon C_1^2C^2(\alpha)+\frac{1}{\varepsilon}\|\tilde{Z}^\alpha\|_{BMO}^2].
\end{eqnarray}
If we choose  $\varepsilon=\frac{C_1}{\nu},$ we derive
$$\|\tilde{Z}^\alpha\|_{BMO}\leq\frac{C_1}{\nu}\sqrt{\nu +TC^2(\alpha)C_1^2}.$$
That is, the norms $\|\tilde{Y}^\alpha\|_\infty$ and $\|\tilde{Z}^\alpha\|_{BMO}$ are uniformly bounded, depending only on $\nu, C_1, C(\alpha)$ and $T$.

Let $\beta$ be a real number to be chosen later, and consider $Y_t^{\alpha,\beta}=e^{\beta t}Y_t^\alpha$ and $\tilde{Y}_t^{\alpha,\beta}=e^{\beta t}
\tilde{Y}_t^\alpha$. By It\^o's formula,
\begin{eqnarray*}
d\tilde{Y}^{\alpha,\beta}=<\tilde{Z}^\alpha, \tilde{K}^\alpha(Y^{\alpha,\beta})>dt+\sqrt{2\nu}<\tilde{Z}^{\alpha,\beta}, dB>
+\beta\tilde{Y}^{\alpha,\beta}dt.
\end{eqnarray*}
Denote $\delta Y^{\alpha,\beta}=Y^{\alpha,\beta}-(Y^{\prime})^{\alpha,\beta}$ and $\delta Z^{\alpha,\beta}=
Z^{\alpha,\beta}-(Z^{\prime})^{\alpha,\beta}$, then we have 
\begin{eqnarray*}
d(\delta\tilde{Y}^{\alpha,\beta})=\Phi^{\alpha,\beta} dt+\beta(\delta\tilde{Y}^{\alpha,\beta})dt+\sqrt{2\nu}<\delta\tilde{Z}^{\alpha,\beta}, dB>
\end{eqnarray*}
where
\begin{eqnarray*}
\Phi_s^{\alpha,\beta}=<\tilde{Z}_s^\alpha, \tilde{K}^\alpha(Y_s^{\alpha,\beta})>-<(\tilde{Z}_s^\prime)^\alpha, \tilde{K}^\alpha((Y_s^{\prime})^{\alpha,\beta})>.
\end{eqnarray*}
According to It\^o's formula
\begin{eqnarray*}
|\delta \tilde{Y}_t^{\alpha,\beta}|^2&=&-2\nu\int_t^T|\delta \tilde{Z}^{\alpha,\beta}|^2 ds
-2\beta\int_t^T|\delta \tilde{Y}^{\alpha,\beta}|^2 ds\\
&&-2\int_t^T(\delta \tilde{Y}^{\alpha,\beta})\Phi^{\alpha,\beta} ds-2\sqrt{2\nu}\int_t^T(\delta\tilde{Y}^{\alpha,\beta})
<\delta \tilde{Z}^{\alpha,\beta}, dB>
\end{eqnarray*}
and taking conditional expectations, we obtain
\begin{eqnarray*}
|\delta \tilde{Y}_t^{\alpha,\beta}|^2=-2\nu\mathbb{E}^{\mathcal{F}_t}\int_t^T|\delta \tilde{Z}^{\alpha,\beta}|^2 ds-2\beta\mathbb{E}^{\mathcal{F}_t}\int_t^T|\delta \tilde{Y}^{\alpha,\beta}|^2 ds
-2\mathbb{E}^{\mathcal{F}_t}\int_t^T(\delta \tilde{Y}^{\alpha,\beta})\Phi^{\alpha,\beta} ds.
\end{eqnarray*}
Integrating over $\mathbb{T}^2$,  
\begin{eqnarray}\label{eq3.11}
\|\delta \tilde{Y}_t^{\alpha,\beta}\|^2&=&-2\nu\mathbb{E}^{\mathcal{F}_t}\int_t^T\|\delta \tilde{Z}^{\alpha,\beta}\|^2 ds-2\beta\mathbb{E}^{\mathcal{F}_t}\int_t^T\|\delta \tilde{Y}^{\alpha,\beta}\|^2 ds\\\nonumber
&&-2\mathbb{E}^{\mathcal{F}_t}\int_t^T\int_{\mathbb{T}^2}(\delta \tilde{Y}^{\alpha,\beta})\Phi^{\alpha,\beta} ds.
\end{eqnarray}
We define
\begin{eqnarray*}
I(t):=\|\delta \tilde{Y}_t^{\alpha,\beta}\|^2+2\nu\mathbb{E}^{\mathcal{F}_t}\int_t^T\|\delta \tilde{Z}^{\alpha,\beta}\|^2 ds+2\beta\mathbb{E}^{\mathcal{F}_t}\int_t^T\|\delta \tilde{Y}^{\alpha,\beta}\|^2 ds.
\end{eqnarray*}
By (\ref{eq3.11}), using Holder's inequality and Young inequality, we obtain
\begin{eqnarray}\label{eq3.12}
I(t)&=&-2\mathbb{E}^{\mathcal{F}_t}\int_t^T\int_{\mathbb{T}^2}(\delta \tilde{Y}^{\alpha,\beta})\Phi^{\alpha,\beta} ds\\\nonumber
&\leq&2\mathbb{E}^{\mathcal{F}_t}\int_t^T\|\delta \tilde{Y}^{\alpha,\beta}\|\|\Phi^{\alpha,\beta} \|ds\\\nonumber
&\leq&2(\mathbb{E}^{\mathcal{F}_t}\int_t^T\|\delta \tilde{Y}^{\alpha,\beta}\|^2ds)^{\frac{1}{2}}(\mathbb{E}^{\mathcal{F}_t}\int_t^T\|\Phi^{\alpha,\beta} \|^2ds)^{\frac{1}{2}}\\\nonumber
&\leq&2\beta\mathbb{E}^{\mathcal{F}_t}\int_t^T\|\delta \tilde{Y}^{\alpha,\beta}\|^2ds
+\frac{1}{2\beta}\mathbb{E}^{\mathcal{F}_t}\int_t^T\|\Phi^{\alpha,\beta} \|^2ds.
\end{eqnarray}
Considering the last integral appearing on the right-hand side of (\ref{eq3.12}), we have
\begin{eqnarray*}
\|\Phi_s^{\alpha,\beta}\|&=&\|\tilde{Z}_s^\alpha\cdot \tilde{K}^\alpha(Y_s^{\alpha,\beta})-(\tilde{Z}_s^{\prime})^\alpha\cdot \tilde{K}^\alpha((Y_s^{\prime})^{\alpha,\beta})\| \\
&=&\|\tilde{Z}_s^\alpha\cdot \tilde{K}^\alpha(\delta Y_s^{\alpha,\beta})+\delta\tilde{Z}_s^{\alpha,\beta}\cdot \tilde{K}^\alpha((Y_s^{\prime})^{\alpha})\| \\
&\leq&\|\tilde{Z}_s^\alpha\| \|\tilde{K}^\alpha(\delta Y_s^{\alpha,\beta})\|+\|\delta\tilde{Z}_s^{\alpha,\beta}\|\|\tilde{K}^\alpha((Y_s^{\prime})^{\alpha})\| \\
&\leq&C(\alpha) \|\tilde{Z}_s^\alpha\| \|\delta Y_s^{\alpha,\beta}\|_{-2,2}+C(\alpha)\|\delta\tilde{Z}_s^{\alpha,\beta}\|\|(Y_s^{\prime})^\alpha\|_{-2,2} 
\end{eqnarray*}
where $\|(Y_s^{\prime})^\alpha\|_{-2,2}\leq C_1$ and 
\begin{eqnarray}\label{eq3.13}
\|\delta Y_s^{\alpha,\beta}\|_{-2,2}^2&=&\int_{\mathbb{R}^2}(1+|\xi|^2)^{-2}|\widehat{\delta Y^{\alpha,\beta}}(\xi)|^2d\xi\\\nonumber
&\leq&\int_{\mathbb{R}^2}|\widehat{\delta Y^{\alpha,\beta}}(\xi)|^2d\xi\\\nonumber
&=&\|\delta Y_s^{\alpha,\beta}\|^2.
\end{eqnarray}
Together with (\ref{eq3.12}) we obtain
\begin{eqnarray}\label{eq3.14}
I(t)&\leq&2\beta\mathbb{E}^{\mathcal{F}_t}\int_t^T\|\delta \tilde{Y}^{\alpha,\beta}\|^2ds
+\frac{C^2(\alpha)}{\beta}\mathbb{E}^{\mathcal{F}_t}\int_t^T\|\tilde{Z}_s^\alpha\|^2 ds\sup\limits_{\Omega\times [0,T]}\|\delta Y_t^{\alpha,\beta}\|^2\\\nonumber
&&+\frac{C^2(\alpha)C_1^2}{\beta}\mathbb{E}^{\mathcal{F}_t}\int_t^T\|\delta\tilde{Z}_s^{\alpha,\beta}\|^2ds\\\nonumber
&\leq&2\beta\mathbb{E}^{\mathcal{F}_t}\int_t^T\|\delta \tilde{Y}^{\alpha,\beta}\|^2ds
+\frac{C^2(\alpha)C_1^2}{\beta}\frac{\nu +TC^2(\alpha)C_1^2}{\nu^2}\sup\limits_{\Omega\times [0,T]}\|\delta Y_t^{\alpha,\beta}\|^2\\\nonumber
&&+\frac{C^2(\alpha)C_1^2}{\beta}\|\delta\tilde{Z}^{\alpha,\beta}\|_{BMO}^2
\end{eqnarray}
where we have used the uniform bounds $\|\tilde{Z}^\alpha\|_{BMO}\leq\frac{C_1}{\nu}\sqrt{\nu +TC^2(\alpha)C_1^2}$ in the second inequality.
Choose $\beta>0$ such that
$$1\geq\frac{\nu}{2},~~~~~~~~~2\nu-\frac{C^2(\alpha)C_1^2}{\beta}\geq \frac{\nu}{2},~~~~~~
\frac{C^2(\alpha)C_1^2}{\beta}\frac{\nu +TC^2(\alpha)C_1^2}{\nu^2}\leq\frac{\nu}{16}.$$
Inequality (\ref{eq3.14}) implies 
$$\sup\limits_{\Omega\times[0,T]}\|\delta\tilde{Y}_t^{\alpha,\beta}\|^2+\|\delta\tilde{Z}^{\alpha,\beta}\|_{BMO}^2\leq\frac{1}{8}\sup\limits_{\Omega\times[0,T]}\|\delta Y_t^{\alpha,\beta}\|^2,$$
that is,
$$\sup\limits_{\Omega\times[0,T]}\|\delta\tilde{Y}_t^{\alpha,\beta}\|+\|\delta\tilde{Z}^{\alpha,\beta}\|_{BMO}\leq\frac{1}{2}\sup\limits_{\Omega\times[0,T]}\|\delta Y_t^{\alpha,\beta}\|.$$

Therefore there exists $\beta>0$ such that, $\mathcal{L}$ is a contraction on $\mathcal{H}$ under the norm
\begin{eqnarray}\label{eq3.15}
\|Y^\alpha\|_{\beta,BMO}=\sup\limits_{\Omega\times[0,T]}\|Y_t^{\alpha,\beta}\|+\|Z^{\alpha,\beta}\|_{BMO},
\end{eqnarray}
where $Z_t^{\alpha,\beta}=e^{\beta t} Z_t^\alpha$ and $Z^\alpha$ is the density process of the martingale part of $Y^\alpha$.

We now complete the proof of Theorem \ref{thm 3.1}. First we construct the sequence of Picard's iteration as  follows.
Begin with 
$$Y_0^\alpha(t,x)=\mathbb{E}\{\xi(x)|\mathcal{F}_t\}$$
and $Z_0^\alpha$  the density process of 
$Y_0^\alpha$ with respect to the Brownian motion determined by It\^o's martingale representation and define 
$Y_{n+1}^\alpha=\mathcal{L}(Y_n^\alpha)$ for $n=0,1,2,\cdots.$ Then  (\ref{eq3.8}) implies that all $Y_n^\alpha\in \mathcal{H}$, so that 
\begin{eqnarray*}
\mathbb{P}\{|Y_n^\alpha(t,x)|\leq C_1  ~~~for~~ all ~~(t,x,n)\in [0,T]\times \mathbb{T}^2\times \mathbb{N}\}=1.
\end{eqnarray*}
Finally, from (\ref{eq3.15}), $\{Y_n^\alpha\}$ is a Cauchy sequence for the norm $\|\cdot\|_{\beta,BMO}$ for some 
$\beta>0$, and therefore it has a limit $Y^\alpha$ which is a solution to $(\ref{eq2.13})$.

The last statement of the Theorem is a well known result in FBSDE equations.

\section{The local existence theorem in $W^{k,p}(\mathbb{R}^d;\mathbb{R}^d)$}\label{section4}
\setcounter{equation}0
\subsection{Formulation of the d-dimensional problem}\label{subsection4.1}

Suppose that $m$ is a smooth solution of (\ref{eq1.2}). We define
$u(t,x):=(I-\alpha^2\Delta)^{-1}m(t,x)$.
Let $W$ be a standard $\mathbb{R}^d$-valued Brownian motion on a complete probability space 
$(\Omega,\mathcal{F},P)$ and $X_s^t(x)$ satisfy the following stochastic differential equation (SDE)
\begin{eqnarray*}
dX_s^t(x)=\sqrt{2\nu}dW_s-u(T-s,X_s^t(x))ds,~~~X_t^t(x)=x,
\end{eqnarray*}
where  $0\leq t\leq s\leq T$.
We denote $Y_s^t(x):= m(T-s,X_s^t(x)), Z_s^t(x):= \nabla m(T-s,X_s^t(x)).$ 
By It\^o's formula, we obtain the following forward backward stochastic differential equation (FBSDE):
\begin{eqnarray}\label{eq4.1}
&&dX_s^t(x)=\sqrt{2\nu}dW_s-u(T-s,X_s^t(x))ds\\\nonumber
&&dY_s^t(x)=\sqrt{2\nu}Z_s^t(x)dW_s+[\nabla p(T-s,X_s^t(x))-\alpha^2(\nabla u^T\Delta u)(T-s,X_s^t(x))]ds\\\nonumber
&&X_t^t(x)=x, Y_T^t(x)=m_0(X_T^t(x))\\\nonumber
&&\Delta p=-\sum_{i,j=1}^d [\partial_i u^j\partial_j (u^i-\alpha^2\Delta u^i)-\alpha^2\partial_{ii}^2u^j\Delta u^j-\alpha^2\partial_iu^j\Delta\partial_iu^j]\\\nonumber
&&u=(I-\alpha^2\Delta)^{-1}m.
\end{eqnarray}
On the other hand, if $(X_s^t(x),Y_s^t(x),Z_s^t(x))$ is a solution of (\ref{eq4.1}), where 
$u$ and $p$ are regular enough as the coefficients,  by Theorem 3.2 in \cite{Pardoux-Peng:92}, the vector field 
$m(t,x):=Y_{T-t}^{T-t}(x)$ satisfies equation (\ref{eq1.2}) for $t\in [0,T]$. From the expression of $\Delta p(t,x)$ in (\ref{eq4.1}) and equation (\ref{eq1.2}), we can derive the divergence free condition $\nabla\cdot m(t)=0$.

For convenience, we use the following notations
\begin{eqnarray}\label{eq4.2}
&&Nf(x):=C(d)\int_{\mathbb{R}^d}\frac{f(y)}{|x-y|^{d-2}}dy,~~ \forall f\in C_c^\infty(\mathbb{R}^d),\\\nonumber
&&G_v:=\sum_{i,j=1}^d [\partial_i v^j\partial_j (v^i-\alpha^2\Delta v^i)-\alpha^2\partial_{ii}^2v^j\Delta v^j-\alpha^2\partial_iv^j\Delta\partial_iv^j],\\\nonumber
&&F_v:=\nabla NG_v, ~~J_v:=F_v+\alpha^2\nabla v^T\Delta v,
\end{eqnarray}
where the operator $N$ satisfies  $\Delta Nf(x)=f(x)$, $\forall f\in C_c^\infty(\mathbb{R}^d)~(d\geq 3)$; $C(d)$ is a constant depending 
on $d$. From \cite{Stein:70}, we know that $Nf$ is well defined for every $f\in L^{p^\prime}(\mathbb{R}^d)$ with $1<p^\prime<\frac{d}{2}$.
Suppose $m_0\in C_c^\infty(\mathbb{R}^d,\mathbb{R}^d),~ m\in C([0,T];C_c^\infty(\mathbb{R}^d,\mathbb{R}^d))$ satisfy $\nabla \cdot m_0=0$ and $\nabla \cdot m(t)=0$ for every $t$. In the following context, we suppress the index $v$ and denote $(X_s^t,Y_s^t,Z_s^t)$ the unique solution of the following FBSDE
\begin{eqnarray}\label{eq4.3}
&&dX_s^t(x)=\sqrt{2\nu}dW_s-v(T-s,X_s^t(x))ds\\\nonumber
&&dY_s^t(x)=\sqrt{2\nu}Z_s^t(x)dW_s-J_v(T-s,X_s^t(x))ds\\\nonumber
&&X_t^t(x)=x, Y_T^t(x)=m_0(X_T^t(x)).
\end{eqnarray}

In the rest of this section, we prove the local existence and uniqueness of the solution to d-dimensional $(d\geq 3)$ Navier-Stokes-$\alpha$ equation (\ref{eq1.2}) through FBSDE (\ref{eq4.3}).

\subsection{Statement of the local existence theorem}\label{subsection4.2}
Let us give the main result for the d-dimensional problem.
\begin{theorem}\label{thm4.1}
Suppose $k>1$ is an integer, $d<p<\infty$, $d\geq 3$ and $m_0\in W^{k,p}(\mathbb{R}^d;\mathbb{R}^d)$ with $\nabla\cdot m_0=0$. Then there exists a vector field $m\in C([0,T_0];W^{k,p}(\mathbb{R}^d;\mathbb{R}^d))$ for some constant $T_0>0$ which only depends on $\|m_0\|_{W^{k,p}}$, such that $m$ is the unique solution of (\ref{eq1.2}).
\end{theorem}

For $1<p<\infty, 1<p^\prime<\frac{d}{2},T>0$, we define
\begin{eqnarray}
&&\mathcal{F}(p,p^\prime,T):=\{ m\in C([0,T];C_c^\infty(\mathbb{R}^d;\mathbb{R}^d)):\\\nonumber
&&\sup\limits_{t\in[0,T]}(\|m(t)\|_{W^{2,p^\prime}}+\|m(t)\|_{W^{k,p}})<\infty,~\forall ~k>1;~~\nabla\cdot m(t)=0,~~\forall t\in [0,T]\},
\end{eqnarray}
\begin{eqnarray}
\mathcal{B}(m_0,T,p,k)&:=&\{ m \in C([0,T];W^{k,p}(\mathbb{R}^d;\mathbb{R}^d)):\nonumber\\
&&m(0,x)=m_0(x),~~\nabla\cdot m(t)=0,~~\forall t\in [0,T]\}.
\end{eqnarray}
For $m\in \mathcal{F}(p,p^\prime,T)$ with
$m_0\in C_c^\infty(\mathbb{R}^d;\mathbb{R}^d)$ and $\nabla \cdot m_0=0$, we define
$$\mathcal{P}_\nu(m):=\mathbb{P}(Y_{T-t}^{T-t}(\cdot))\in C([0,T];C_c^\infty(\mathbb{R}^d;\mathbb{R}^d)),$$
where $Y$ is the solution of (\ref{eq4.3}) with coefficients $v:=(I-\alpha^2\Delta)^{-1}m$ and initial condition $m_0:=m(0)$, $\mathbb{P}$ the Leray-Hodge projection on the space of divergence free vector fields. 
To prove Theorem \ref{thm4.1}, we first extend the map $\mathcal{P}_\nu$ from $\mathcal{F}(p,p^\prime,T)$ to $\mathcal{B}(m_0,T,p,k)$ in Proposition \ref{prop4.7}. Then we prove that there is a unique fixed point of the map $\mathcal{P}_{\nu}:\mathcal{B}(m_0,T,p,k)\rightarrow \mathcal{B}(m_0,T,p,k)$ in Theorem \ref{thm4.8}. At last, we use the fixed point theorem to show the local existence and uniqueness of the solution to the Navier-Stokes-$\alpha$ equation.

\subsection{Proof of the local existence theorem}\label{subsection4.3}
In this subsection, we use the method in \cite{Chen-Cruzeiro-Qian:13} to prove Theorem \ref{thm4.1}. The constant C may be different according to the context.
\begin{lemma}\label{lemma4.2}
Let $d<p<\infty$ and $1<p^\prime<\frac{d}{2}$. Then for $m:=v-\alpha^2\Delta v\in W^{2,p^\prime}(\mathbb{R}^d;\mathbb{R}^d)$,$~~m\in W^{k,p}(\mathbb{R}^d;\mathbb{R}^d)~(k>1)$ and $\nabla\cdot m(t)=0$, we have 
\begin{eqnarray*}
&&\|J_v\|_{L^p}\leq C\|\nabla m\|_{L^\infty}\|v\|_{L^p}+C\alpha^2\|v\|_{W^{2,p}}\|v\|_{W^{2,p}}\\
&&\|J_v\|_{W^{1,p}}\leq C\|\nabla m\|_{L^\infty}\|v\|_{W^{1,p}}+C\alpha^2\|v\|_{W^{2,p}}\|v\|_{W^{3,p}}\\
&&\|J_v\|_{W^{k,p}}\leq C(\alpha)\|m\|_{W^{k,p}}\|v\|_{W^{k,p}},
\end{eqnarray*}
where $C(\alpha)$ is a constant which only depends on $\alpha$.
\end{lemma}
\begin{proof}
Since $\nabla\cdot m=0$, we have
\begin{eqnarray*}
G_v&=&\sum_{i,j=1}^d [\partial_i v^j\partial_j (v^i-\alpha^2\Delta v^i)-\alpha^2\partial_{ii}^2v^j\Delta v^j-\alpha^2\partial_iv^j\Delta\partial_iv^j]\\
&=&\sum_{i=1}^d\partial_i\sum_{j=1}^d [v^j\partial_j(v^i-\alpha^2\Delta v^i)-\alpha^2\partial_iv^j\Delta v^j]\\
&:=&\sum_{i=1}^d\partial_i f_i.
\end{eqnarray*}
Therefore, $F_v=\nabla NG_v=\sum_{i=1}^d\nabla N\partial_i f_i$. By Chapter 2 in \cite{Stein:70}, we can check that $\nabla N\partial_i$ is a singular integral operator and that it is bounded in $L^p~ (1< p<\infty)$ space. Then we have
\begin{eqnarray*}
\|F_v\|_{L^p}\leq C\sum_{i=1}^d\|f_i\|_{L^p}
&\leq& C\sum_{i,j=1}^d[\|v^j\partial_j(v^i-\alpha^2\Delta v^i)\|_{L^p}+\alpha^2\|\partial_iv^j\Delta v^j\|_{L^p}]\\
&\leq& C(\|\nabla m\|_{L^\infty}\|v\|_{L^p}+\alpha^2\|\nabla v\|_{L^\infty}\|\Delta v\|_{L^p}).
\end{eqnarray*}
Since $J_v:=F_v+\alpha^2\nabla v^T\Delta v$, by the Sobolev embedding theorem $(d<p<\infty)$, we have
\begin{eqnarray*}
\|J_v\|_{L^p}&\leq&C(\|F_v\|_{L^p}+\alpha^2\|\nabla v^T\Delta v\|_{L^p})\\
&\leq& C(\|\nabla m\|_{L^\infty}\|v\|_{L^p}+\alpha^2\|\nabla v\|_{L^\infty}\|\Delta v\|_{L^p})\\
&\leq&C(\|\nabla m\|_{L^\infty}\|v\|_{L^p}+\alpha^2\|v\|_{W^{2,p}}\|v\|_{W^{2,p}}).
\end{eqnarray*}

Since $\nabla F_v=\nabla^2 NG_v$ and $\nabla^2 N$ is a singular integral operator, by the Sobolev embedding theorem, we get
\begin{eqnarray*}
\|\nabla F_v\|_{L^p}\leq C\|G_v\|_{L^p}
&\leq& C(\| \partial_i v^j\partial_j (v^i-\alpha^2\Delta v^i)\|_{L^p}+\alpha^2\|\partial_{ii}^2v^j\Delta v^j\|_{L^p}+\alpha^2\|\partial_iv^j\Delta\partial_iv^j\|_{L^p})\\
&\leq& C(\|\nabla m\|_{L^\infty}\|v\|_{W^{1,p}}+\alpha^2\|v\|_{W^{2,p}}\|v\|_{W^{3,p}}).
\end{eqnarray*}
Therefore, 
\begin{eqnarray*}
\|\nabla J_v\|_{L^p}&\leq& C(\|\nabla F_v\|_{L^p}+\alpha^2\|\nabla(\nabla v^T\Delta v)\|_{L^p})\\
&\leq&C(\|\nabla m\|_{L^\infty}\|v\|_{W^{1,p}}+\alpha^2\|v\|_{W^{2,p}}\|v\|_{W^{3,p}}).
\end{eqnarray*}

Using the same procedure, we obtain
\begin{eqnarray*}
\|\nabla^2 F_v\|_{L^p}&\leq& C\|\nabla G_v\|_{L^p}\\
&\leq&C(\|\nabla\partial_i v^j\partial_j (v^i-\alpha^2\Delta v^i) \|_{L^p} +\|\partial_i v^j\nabla\partial_j (v^i-\alpha^2\Delta v^i) \|_{L^p}\\
&&+\alpha^2\|\nabla(\partial_{ii}^2v^j\Delta v^j)\|_{L^p}+\alpha^2\|\nabla(\partial_iv^j\Delta\partial_iv^j)\|_{L^p}\\
&\leq&C(\|\nabla m\|_{L^\infty}\|v\|_{W^{2,p}}+\|m\|_{W^{2,p}}\|\nabla v\|_{L^\infty}\\
&&+\alpha^2\|v\|_{W^{2,p}}\|v\|_{W^{4,p}})\\
&\leq&C\|m\|_{W^{2,p}}\|v\|_{W^{2,p}}+C\alpha^2\|v\|_{W^{2,p}}\|(I-\alpha^2\Delta)^{-1}m\|_{W^{4,p}}\\
&\leq&C(\alpha)\|m\|_{W^{2,p}}\|v\|_{W^{2,p}}.
\end{eqnarray*}
Therefore,
\begin{eqnarray*}
\|\nabla^2 J_v\|_{L^p}&\leq&C(\|\nabla^2 F_v\|_{L^p}+\alpha^2\|\nabla^2(\nabla v^T\Delta v)\|_{L^p})\\
&\leq&C(\alpha)\|m\|_{W^{2,p}}\|v\|_{W^{2,p}}.
\end{eqnarray*}

Repeating the procedure above, we get
$$\|\nabla^k J_v\|_{L^p}\leq C(\alpha)\|m\|_{W^{k,p}}\|v\|_{W^{k,p}}.$$
\end{proof}

Applying Lemma \ref{lemma4.2}, we obtain the $W^{k,p}$ bound for the solution of (\ref{eq4.3}).
\begin{lemma}\label{lemma4.3}
Suppose $m\in \mathcal{F}(p,p^\prime,T)$ where $d<p<\infty, 1<p^\prime<\frac{d}{2}, T>0$. Let $(X,Y,Z)$ be the unique solution of (\ref{eq4.3}) with coefficients $v=(I-\alpha^2\Delta)^{-1}m$ and initial condition $m_0:=m(0)$. Denote $\Phi(t,x)=Y_t^t(x)$, then
we have 
\begin{eqnarray}\label{eq4.6}
\sup\limits_{t\in [0, T]}\|\Phi(t)\|_{W^{k,p}}\leq Ce^{CKT}\|m_0\|_{W^{k,p}}(1+TK)^{k-1}+C(\alpha)e^{CKT}TKK_0(1+TK)^{k-1},
\end{eqnarray}
where $k>1$, $K_0:=\sup\limits_{t\in[0,T]}\|m(t)\|_{W^{k,p}}, K:=\sup\limits_{t\in[0,T]}\|v(t)\|_{W^{k,p}}$, $C(\alpha)$ is a constant which only depends on $\alpha$.
\end{lemma}
\begin{proof}
We first prove the following assertion:
for every $f\in W^{k,p}(\mathbb{R}^d)$ with $k>1,~d<p<\infty$, we have 
\begin{eqnarray}\label{eq4.7}
\sup\limits_{0\leq t\leq s\leq T}\|f\circ X_s^t\|_{W^{k,p}}\leq Ce^{CKT}(1+TK)^{k-1}\|f\|_{W^{k,p}} a.s.,
\end{eqnarray} 
where $f\circ X_s^t$ denotes the composition of $f$ and $X_s^t$.

Since $\nabla\cdot m(t)=0$, we have $\nabla\cdot u(t)=\nabla\cdot (I-\alpha^2\Delta)^{-1}m(t)=
(I-\alpha^2\Delta)^{-1}\nabla\cdot m(t)=0$. 
Hence for every $h\in L^1({\mathbb{R}^d})$, 
\begin{eqnarray}\label{eq4.8}
\int_{\mathbb{R}^d} h(X_s^t(x))dx=\int_{\mathbb{R}^d}h(x)dx. ~~a.s.
\end{eqnarray} 
Let us take $h=|f|^p$ in (\ref{eq4.8}),
\begin{eqnarray}\label{eq4.9}
\int_{\mathbb{R}^d} |f(X_s^t(x))|^pdx=\int_{\mathbb{R}^d}|f(x)|^pdx=\|f\|_{L^p}^p. ~~a.s.
\end{eqnarray} 

Since $m\in \mathcal{F}(p,p^\prime,T)$, there is a $C^\infty$-differentiable version of $X_s^t(\cdot)$ and $\nabla X_s^t, \nabla^2 X_s^t$ satisfying the following equation:
\begin{eqnarray}\label{eq4.10}
&&d\nabla X_s^t(x)=-\nabla v(T-s,X_s^t(x))\nabla X_s^t(x)ds\\\nonumber
&&d\nabla^2 X_s^t(x)=-\nabla v(T-s,X_s^t(x))\nabla^2 X_s^t(x)ds-\nabla^2 v(T-s,X_s^t(x))(\nabla X_s^t(x))^2ds\\\nonumber
&&\nabla X_t^t(x)=id, ~~\nabla^2 X_t^t(x)=0,
\end{eqnarray} 
where $id$ denotes the identity map in $\mathbb{R}^d$. 
By the Sobolev embedding theorem ($d<p<\infty$),
$$\|v(t)\|_{L^\infty}\leq C\|v(t)\|_{W^{1,p}}\leq CK,$$      
$$\|\nabla v(t)\|_{L^\infty}\leq C\|\nabla v(t)\|_{W^{1,p}} \leq C\|v(t)\|_{W^{2,p}}\leq CK.$$ From Gr\"onwall's inequality, we deduce that
\begin{eqnarray}\label{eq4.11}
|\nabla X_s^t(x)|\leq Ce^{CKT}~~a.s., ~~\forall x\in \mathbb{R}^d.
\end{eqnarray}
For $f\in W^{2,p}(\mathbb{R}^d)\cap C^1(\mathbb{R}^d)$, we have $\nabla (f\circ X_s^t)(x)=\nabla f(X_s^t(x))\nabla X_s^t(x)$.
By (\ref{eq4.8}) and (\ref{eq4.11}), we obtain
\begin{eqnarray}\label{eq4.12}
\int_{\mathbb{R}^d}|\nabla (f\circ X_s^t)(x)|^pdx\leq Ce^{CKT}\|\nabla f\|_{L^p}^p. ~a.s.
\end{eqnarray}

By (\ref{eq4.10}) and (\ref{eq4.11}),
\begin{eqnarray*}
|\nabla^2X_s^t(x)|\leq CK\int_t^s |\nabla^2X_r^t(x)|dr+Ce^{CKT}\int_t^s|\nabla^2 v(T-r,X_r^t(x))|dr;
\end{eqnarray*}
applying Gr\"onwall's inequality, we obtain
\begin{eqnarray*}
|\nabla^2X_s^t(x)|\leq Ce^{CKT}\int_t^s|\nabla^2 v(T-r,X_r^t(x))|dr,
\end{eqnarray*}
together with (\ref{eq4.8}) and H\"older's inequality,
\begin{eqnarray}\label{eq4.13}
&&\int_{\mathbb{R}^d}|\nabla^2X_s^t(x)|^pdx\leq CT^{p-1}e^{CKT}\int_t^s\int_{\mathbb{R}^d}|\nabla^2 v(T-r,X_r^t(x))|^pdxdr\\\nonumber
&&\leq CT^pe^{CKT}\sup\limits_{t\in [0,T]}\|v(t)\|_{W^{2,p}}^p.
\end{eqnarray}
For $f\in C^2(\mathbb{R}^d)\bigcap W^{2,p}(\mathbb{R}^d)$,
\begin{eqnarray}\label{eq4.14}
\nabla^2(f\circ X_s^t)(x)=\nabla^2f(X_s^t(x))(\nabla X_s^t(x))^2+\nabla f(X_s^t(x))\nabla^2X_s^t(x),
\end{eqnarray}
by (\ref{eq4.8}) (\ref{eq4.11}) and (\ref{eq4.13}),
\begin{eqnarray}\label{eq4.15}
&&\int_{\mathbb{R}^d}|\nabla^2 (f\circ X_s^t)(x)|^pdx\\\nonumber
&&\leq C\|\nabla X_s^t(\cdot)\|_{L^\infty}^{2p}\int_{\mathbb{R}^d}|\nabla^2f(X_s^t(x))|^pdx
+C\|\nabla f\|_{L^\infty}^p\int_{\mathbb{R}^d}|\nabla^2X_s^t(x)|^pdx\\\nonumber
&&\leq Ce^{CKT}\|\nabla^2 f\|_{L^p}^p+CT^pK^pe^{CKT}\|\nabla f\|_{W^{1,p}}^p\\\nonumber
&&\leq C(1+T^pK^p)e^{CKT}\|f\|_{W^{2,p}}^p.
\end{eqnarray}

For general $f\in W^{2,p}(\mathbb{R}^d)$, we can choose a sequence $\{f_n\}_{n=1}^{\infty}\subset C^2(\mathbb{R}^d)\bigcap W^{2,p}(\mathbb{R}^d)$, such that $\lim\limits_{n\rightarrow \infty}\|f_n-f\|_{W^{2,p}}=0$. By approximation procedure, we know 
$f\circ X_s^t \in W^{2,p}(\mathbb{R}^d)$, and (\ref{eq4.12}), (\ref{eq4.15}) still hold.

Repeating the procedure above, we can also obtain the estimate (\ref{eq4.7}).
Up to now, the assertion has been proved.

In the following context, we will use the assertion to prove our result. 
Since $m\in \mathcal{F}(p,p^\prime,T)$, from 
the computations in \cite{Pardoux-Peng:92}, we know that, for every $q\geq 2$,
\begin{eqnarray*}
\mathbb{E}(\sup\limits_{0\leq t\leq s\leq T}(|Y_s^t(x)|^q+|Z_s^t(x)|^q))<\infty.
\end{eqnarray*}
Since $Y_t^t(x)$ is deterministic, taking expectation in (\ref{eq4.3}), we obtain the Feynmann-Kac 
formula:
\begin{eqnarray}\label{eq4.16}
\Phi(t,x)=Y_t^t(x)=\mathbb{E}(m_0(X_T^t(x)))+\int_t^T\mathbb{E}J_v(T-s, X_s^t(x))ds.
\end{eqnarray}
By Lemma \ref{lemma4.2}, $J_v(t)\in W^{k,p}(\mathbb{R}^d;\mathbb{R}^d)\bigcap C_c^\infty(\mathbb{R}^d;\mathbb{R}^d)$ for every $t$. Therefore, we can change the order of expectation and differential, and  
apply H\"older's inequality:
\begin{eqnarray*}
\|\mathbb{E}J_v(T-s,X_s^t(\cdot))\|_{W^{k,p}}^p\leq \mathbb{E}\|J_v(T-s,X_s^t(\cdot))\|_{W^{k,p}}^p.
\end{eqnarray*}
By Lemma \ref{lemma4.2} and (\ref{eq4.7}), 
\begin{eqnarray*}
&&\|\mathbb{E}J_v(T-s,X_s^t(\cdot))\|_{W^{k,p}}\\
&\leq& Ce^{CKT}(1+TK)^{k-1}\|J_v(T-s)\|_{W^{k,p}}\\
&\leq&C(\alpha)e^{CKT}(1+TK)^{k-1}KK_0.
\end{eqnarray*}
Similarly, 
\begin{eqnarray*}
\|\mathbb{E}(m_0(X_T^t(\cdot)))\|_{W^{k,p}}
\leq Ce^{CKT}(1+TK)^{k-1}\|m_0\|_{W^{k,p}}.
\end{eqnarray*}
Putting the above estimate into (\ref{eq4.16}), we get (\ref{eq4.6}).
\end{proof}

\begin{lemma}\label{lemma4.4}
Suppose $m_l\in \mathcal{F}(p,p^\prime,T),l=1,2, ~v_l=(I-\alpha^2\Delta)^{-1}m_l$, for some $d<p<\infty, 1<p^\prime<\frac{d}{2},T>0$. We have the following estimates 
\begin{eqnarray*}
\|J_{v_1(t)}-J_{v_2(t)}\|_{W^{k-1,p}}\leq  C(\alpha)(K_0\|v_1-v_2\|_{W^{k-1,p}}+K\|m_1-m_2\|_{W^{k-1,p}}),
\end{eqnarray*}
where $k>1, ~K_0:=\sup\limits_{t\in[0,T],l=1,2}\|m_l\|_{W^{k,p}}, K:=\sup\limits_{t\in[0,T],l=1,2}\|v_l\|_{W^{k,p}}$, $C(\alpha)$ is a constant which depends on $\alpha$.
\end{lemma}
\begin{proof}
Since $F_{v_1(t)}-F_{v_2(t)}=\sum_{i=1}^d\nabla N\partial_i 
(f_{i,1}(t)-f_{i,2}(t))$, where $f_{i,l}:=\sum_{j=1}^d[v_l^j\partial_j(v_l^i-\alpha^2\Delta v_l^i)-\alpha^2\partial_iv_l^j\Delta v_l^j]$, 
by the Sobolev embedding theorem, we obtain
\begin{eqnarray*}
&&\|F_{v_1(t)}-F_{v_2(t)}\|_{L^p}\leq C\sum_{i=1}^d\|f_{i,1}(t)-f_{i,2}(t)\|_{L^p}\\
&\leq& C(\|v_1^j\partial_jm_1^i-v_2^j\partial_jm_2^i\|_{L^p}+\alpha^2
\|\partial_iv_1^j\Delta v_1^j-\partial_iv_2^j\Delta v_2^j\|_{L^p})\\
&\leq& C(\|v_1-v_2\|_{L^p}\|\nabla m_1\|_{L^\infty}+\|v_2\|_{L^\infty}
\|\nabla m_1-\nabla m_2\|_{L^p}\\
&&+\alpha^2\|\nabla v_1-\nabla v_2\|_{L^\infty}\|\Delta v_1\|_{L^p}+\alpha^2\|\nabla v_2\|_{L^\infty}\|\Delta (v_1-v_2)\|_{L^p})\\
&\leq& C(\alpha)(\sup\limits_{l=1,2}\|m_l\|_{W^{2,p}}\|v_1-v_2\|_{L^p}+\sup\limits_{l=1,2}\|v_l\|_{W^{1,p}}\|\nabla m_1-\nabla m_2\|_{L^p}\\
&&+\sup\limits_{l=1,2}\|v_l\|_{W^{2,p}}\|v_1- v_2\|_{W^{2,p}}).
\end{eqnarray*}
Therefore,
\begin{eqnarray*}
\|J_{v_1}-J_{v_2}\|_{L^p}&\leq&C(\|F_{v_1}-F_{v_2}\|_{L^p}+\alpha^2\|\nabla v_1^T\Delta v_1-\nabla v_2^T\Delta v_2\|_{L^p})\\
&\leq& C(\alpha)(\sup\limits_{l=1,2}\|m_l\|_{W^{2,p}}\|v_1-v_2\|_{L^p}+\sup\limits_{l=1,2}\|v_l\|_{W^{1,p}}\|\nabla m_1-\nabla m_2\|_{L^p}\\
&&+\sup\limits_{l=1,2}\|v_l\|_{W^{2,p}}\|v_1- v_2\|_{W^{2,p}}).
\end{eqnarray*}

Since $\nabla (F_{v_1(t)}-F_{v_2(t)})=\nabla^2N(G_{v_1(t)}-G_{v_2(t)})$, we obtain
\begin{eqnarray*}
&&\|\nabla(F_{v_1(t)}-F_{v_2(t)})\|_{L^p}\leq C\|G_{v_1(t)}-G_{v_2(t)})\|_{L^p}\\
&\leq& C(\|\partial_iv_1^j\partial_j m_1^i-\partial_iv_2^j\partial_j m_2^i\|_{L^p}\\
&&+\alpha^2\|\Delta v_1^j\Delta v_1^j-\Delta v_2^j\Delta v_2^j\|_{L^p}+\alpha^2
\|\partial_i v_1^j\Delta \partial_i v_1^j-\partial_i v_2^j\Delta \partial_i v_2^j\|_{L^p})\\
&\leq& C(\alpha)(\|\nabla v_1-\nabla v_2\|_{L^p}\|\nabla m_1\|_{L^\infty}+\|\nabla v_2\|_{L^\infty}\|\nabla m_1-\nabla m_2\|_{L^p}\\
&&+\sup\limits_{l=1,2}\|v_l\|_{W^{2,p}}\|v_1-v_2\|_{W^{3,p}}+\sup\limits_{l=1,2}\|v_l\|_{W^{4,p}}\|v_1-v_2\|_{W^{1,p}})\\
&\leq& C(\alpha)(\sup\limits_{l=1,2}\|m_l\|_{W^{2,p}}\|v_1-v_2\|_{W^{1,p}}+\sup\limits_{l=1,2}\|v_l\|_{W^{2,p}}\|m_1-m_2\|_{W^{1,p}}).
\end{eqnarray*}
Therefore, 
\begin{eqnarray*}
&&\|\nabla(J_{v_1}-J_{v_2})\|_{L^p}\leq C(\|\nabla(F_{v_1}-F_{v_2})\|_{L^p}+\alpha^2\|\nabla(\nabla v_1^T\Delta v_1-\nabla v_2^T\Delta v_2)\|_{L^p})\\
&\leq&C(\alpha)(\sup\limits_{l=1,2}\|m_l\|_{W^{2,p}}\|v_1-v_2\|_{W^{1,p}}+\sup\limits_{l=1,2}\|v_l\|_{W^{2,p}}\|m_1-m_2\|_{W^{1,p}}).
\end{eqnarray*}

Repeating the procedure above, we get
\begin{eqnarray*}
&&\|\nabla^{(k-1)}(J_{v_1}-J_{v_2})\|_{L^p}\\
&\leq& C(\alpha)(\sup\limits_{l=1,2}\|m_l\|_{W^{k,p}}\|v_1-v_2\|_{W^{k-1,p}}+\sup\limits_{l=1,2}\|v_l\|_{W^{k,p}}\|m_1-m_2\|_{W^{k-1,p}}).
\end{eqnarray*}
\end{proof}

For vector fields $m_l\in \mathcal{F}(p,p^\prime, T)$, where $d<p<\infty, 1<p^\prime<\frac{d}{2}, 0<T<1,l=1,2$,
let $(X_l,Y_l,Z_l)$ be the solutions of (\ref{eq4.3}) with coefficients $v_l=(I-\alpha^2\Delta)^{-1}m_l$ and initial condition $m_{0,l}:=m_l(0)$. We use Lemma \ref{lemma4.2} and Lemma \ref{lemma4.4} to prove the following Lemma.

\begin{lemma}\label{lemma4.5}
Let $\Phi_l(t,x):=Y_{t,l}^t(x),l=1,2$, and $k>1,~d<p<\infty, 0\leq t\leq T$, we have the following estimate 
\begin{eqnarray}\label{eq4.17}
&&\sup\limits_{t\in [0, T]}\|\Phi_1(t)-\Phi_2(t)\|_{W^{k-1,p}}\\\nonumber
&\leq& Ce^{CKT}(1+TK)^{k-2}\|m_{0,1}-m_{0,2}\|_{W^{k-1,p}}\\\nonumber
&&+C(\alpha)TK_0(1+TK)^{k}e^{CKT}\sup_{t\in[0,T]}\|v_1(t)-v_2(t)\|_{W^{k-1,p}}\\\nonumber
&&+C(\alpha)TKe^{CKT}(1+TK)^{k-2}\sup_{t\in[0,T]}\|m_1(t)-m_2(t)\|_{W^{k-1,p}},
\end{eqnarray}
where $K_0:=\sup\limits_{t\in[0,T],l=1,2}\|m_l(t)\|_{W^{k,p}}, K:=\sup\limits_{t\in[0,T],l=1,2}\|v_l(t)\|_{W^{k,p}}$, $C(\alpha)$ is a constant which depends on $\alpha$.
\end{lemma}
\begin{proof}
We first prove the following assertion:
for every $f_1,f_2\in W^{k,p}(\mathbb {R}^d)$ with $k>1$ and $d<p<\infty$, we have 
 \begin{eqnarray}\label{eq4.18}
&&\|(f_1\circ X_{s,1}^t)(\cdot)-(f_2\circ X_{s,2}^t)(\cdot)\|_{W^{k-1,p}}\\\nonumber
&\leq&Ce^{CKT}(1+TK)^{k-2}\|f_1-f_2\|_{W^{k-1,p}}\\\nonumber
&&+CTe^{CKT}(1+TK)^{k-1}\| f_2\|_{W^{k,p}}\sup\limits_{t\in[0,T]}\|v_1(t)-v_2(t)\|_{W^{k-1,p}}~a.s.
\end{eqnarray}

We first consider the case $f_1,f_2\in C^1(\mathbb{R}^d)\bigcap W^{1,p}(\mathbb{R}^d)$. 
By triangle inequality,
\begin{eqnarray*}
|f_1(X_{s,1}^t)-f_2(X_{s,2}^t)|\leq |f_1(X_{s,1}^t)-f_2(X_{s,1}^t)|+|f_2(X_{s,1}^t)-f_2(X_{s,2}^t)|.
\end{eqnarray*}
By (\ref{eq4.8}), we deduce that
\begin{eqnarray}\label{eq4.19}
\int_{\mathbb{R}^d}|f_1(X_{s,1}^t(x))-f_2(X_{s,1}^t(x))|^p dx\leq \|f_1-f_2\|_{L^p}^p. a.s.
\end{eqnarray}
Let $X_{s}^{t,r}(x), ~1\leq r\leq 2$ satisfy the following SDE:
\begin{eqnarray*}
&&dX_{s}^{t,r}(x)=\sqrt{2\nu}dW_s-((2-r)v_1(T-s,X_{s}^{t,r}(x))+(r-1)v_2(T-s,X_{s}^{t,r}(x)))ds,\\
&&X_{t}^{t,r}(x)=x,~~0\leq t\leq s\leq T.
\end{eqnarray*}
Obviously, $X_{s}^{t,r}(x)|_{r=1}=X_{s,1}^t(x)$ and $X_{s}^{t,r}(x)|_{r=2}=X_{s,2}^t(x)$. Since 
$\nabla \cdot ((2-r)v_1+(r-1)v_2)=0$, we obtain 
\begin{eqnarray}\label{eq4.20}
\int_{\mathbb{R}^d}h(X_{s}^{t,r}(x))dx=\int_{\mathbb{R}^d}h(x)dx, ~~\forall r\in [1,2],~h\in L^1(\mathbb{R}^d) ~~a.s.
\end{eqnarray}
Since $m_l\in\mathcal{F}(p,p^\prime,T)$, there is a differentiable version of $X_{s}^{t,r}(x)$ with respect to $r$ (c.f. \cite{Kunita:90}). Let $V_s^{t,r}(x):=\frac{d}{dr}X_s^{t,r}(x)$, then it satisfies the following SDE:
\begin{eqnarray*}
dV_{s}^{t,r}(x)&=&-((2-r)\nabla v_1(T-s,X_{s}^{t,r}(x))+(r-1)\nabla v_2(T-s,X_{s}^{t,r}(x)))V_{s}^{t,r}(x)ds\\
&&+( v_1(T-s,X_{s}^{t,r}(x))- v_2(T-s,X_{s}^{t,r}(x)))ds,\\
V_{t}^{t,r}(x)&=&0,~~0\leq t\leq s\leq T.
\end{eqnarray*}
Hence by Gr\"onwall's inequality, for every $r\in [1,2]$ and $x\in\mathbb{R}^d$, 
\begin{eqnarray}\label{eq4.21}
|V_s^{t,r}(x)|\leq CTe^{CKT}\sup\limits_{t\in[0,T]}\|v_1(t)-v_2(t)\|_{L^\infty}. a.s.
\end{eqnarray}
Since
\begin{eqnarray*}
|f_2(X_{s,1}^t(x))-f_2(X_{s,2}^t(x))|^p= |\int_1^2 \frac{d}{dr}(f_2(X_s^{t,r}(x)))dr|^p\leq \int_1^2 |\nabla f_2(X_s^{t,r}(x))|^p
|V_s^{t,r}(x)|^pdr,
\end{eqnarray*}
by (\ref{eq4.20}) and (\ref{eq4.21}), 
\begin{eqnarray}\label{eq4.22}
&&\int_{\mathbb{R}^d}|f_2(X_{s,1}^t(x))-f_2(X_{s,2}^t(x))|^pdx\\\nonumber
&\leq&CT^p e^{CKT}\sup\limits_{t\in[0,T]}\|v_1(t)-v_2(t)\|_{L^\infty}^p\int_1^2\int_{\mathbb{R}^d}|\nabla f_2(X_s^{t,r}(x))|^pdxdr\\\nonumber
&\leq&CT^p e^{CKT}\sup\limits_{t\in[0,T]}\|v_1(t)-v_2(t)\|_{L^\infty}^p\int_{\mathbb{R}^d}|\nabla f_2(x)|^pdx. a.s.
\end{eqnarray}
Applying (\ref{eq4.19}) and (\ref{eq4.22}), by the Sobolev embedding theorem, we deduce that 
\begin{eqnarray}\label{eq4.23}
&&\int_{\mathbb{R}^d}|f_1(X_{s,1}^t(x))-f_2(X_{s,2}^t(x))|^pdx\\\nonumber
&\leq&C\|f_1-f_2\|_{L^p}^p+CT^pe^{CKT}\|\nabla f_2\|_{L^p}^p\sup\limits_{t\in[0,T]}\|v_1(t)-v_2(t)\|_{L^\infty}^p\\\nonumber
&\leq&C\|f_1-f_2\|_{L^p}^p+CT^pe^{CKT}\|f_2\|_{W^{1,p}}^p\sup\limits_{t\in[0,T]}\|v_1(t)-v_2(t)\|_{W^{1,p}}^p~a.s.
\end{eqnarray}

For general $f_1,f_2\in W^{1,p}(\mathbb{R}^d)$, there exist sequences $\{f_{1,n}\}_{n=1}^{\infty},\{f_{2,n}\}_{n=1}^{\infty}
\subset C^1(\mathbb{R}^d)\bigcap W^{1,p}(\mathbb{R}^d)$, such that 
\begin{eqnarray*}
\lim_{n\rightarrow \infty}\sup\limits_{x\in\mathbb{R}^d}|f_{i,n}(x)-f_i(x)|\leq \lim_{n\rightarrow \infty}\|f_{i,n}-f_i\|_{W^{1,p}}=0,
~~\sup_{n}\|f_{i,n}\|_{W^{1,p}}\leq \|f_i\|_{W^{1,p}}, i=1,2.
\end{eqnarray*}
By Fatou's lemma,
\begin{eqnarray*}
&&\int_{\mathbb{R}^d}|f_1(X_{s,1}^t(x))-f_2(X_{s,2}^t(x))|^pdx=\int_{\mathbb{R}^d}\lim\limits_{n\rightarrow \infty}|f_{1,n}(X_{s,1}^t(x))-f_{2,n}(X_{s,2}^t(x))|^pdx\\\nonumber
&\leq&\liminf\limits_{n\rightarrow \infty}\int_{\mathbb{R}^d}|f_{1,n}(X_{s,1}^t(x))-f_{2,n}(X_{s,2}^t(x))|^pdx\\\nonumber
&\leq&C\|f_1-f_2\|_{L^p}^p+CT^pe^{CKT}\|f_2\|_{W^{1,p}}^p\sup\limits_{t\in[0,T]}\|v_1(t)-v_2(t)\|_{W^{1,p}}^p~a.s.
\end{eqnarray*}

Since $\nabla (f_l\circ X_{s,l}^t)(x)=\nabla f_l(X_{s,l}^t(x))\nabla X_{s,l}^t(x)$, by (\ref{eq4.11}), we obtain
\begin{eqnarray}\label{eq4.24}
&&|\nabla (f_1\circ X_{s,1}^t)(x)-\nabla (f_2\circ X_{s,2}^t)(x)|\\\nonumber
&\leq&Ce^{CKT}|\nabla f_1 (X_{s,1}^t(x))-\nabla f_2( X_{s,2}^t(x))|+C\|\nabla f_2\|_{L^\infty}|\nabla X_{s,1}^t(x)-\nabla X_{s,2}^t(x)|.
\end{eqnarray}
According to  (\ref{eq4.23}), we deduce that
\begin{eqnarray*}
&&\int_{\mathbb{R}^d}|\nabla f_1 (X_{s,1}^t(x))-\nabla f_2( X_{s,2}^t(x))|^pdx\\\nonumber
&\leq&C\|f_1-f_2\|_{W^{1,p}}^p+CT^pe^{CKT}\| f_2\|_{W^{2,p}}^p\sup\limits_{t\in[0,T]}\|v_1(t)-v_2(t)\|_{W^{1,p}}^p~a.s.
\end{eqnarray*}
Denote $\Gamma_s^t(x):= \nabla X_{s,1}^t(x)-\nabla X_{s,2}^t(x)$, by (\ref{eq4.11}), hence
\begin{eqnarray*}
|\Gamma_s^t(x)|\leq CK\int_t^s |\Gamma_r^t(x)| dr+Ce^{CKT}\int_t^s |\nabla v_1(T-r,X_{r,1}^t(x))-\nabla v_2(T-r, X_{r,2}^t(x))|dr.
\end{eqnarray*}
Substituting $f_1, f_2$ with $\nabla v_1, \nabla v_2$ in (\ref{eq4.23}), together with Gr\"onwall's inequality and H\"older's inequality, we have
\begin{eqnarray*}
\int_{\mathbb{R}^d}|\Gamma_s^t(x)|^pdx
\leq CT^pe^{CKT}(1+T^pK^p)\sup\limits_{t\in[0,T]}\|v_1(t)-v_2(t)\|_{W^{1,p}}^p ~a.s.
\end{eqnarray*}
Putting the above estimates into (\ref{eq4.24}), 
\begin{eqnarray}\label{eq4.25}
&&\int_{\mathbb{R}^d}|\nabla (f_1\circ X_{s,1}^t)(x)-\nabla (f_2\circ X_{s,2}^t)(x)|^pdx\\\nonumber
&\leq&Ce^{CKT}\|f_1-f_2\|_{W^{1,p}}^p+CT^pe^{CKT}(1+T^pK^p)\| f_2\|_{W^{2,p}}^p\sup\limits_{t\in[0,T]}\|v_1(t)-v_2(t)\|_{W^{1,p}}^p~a.s.
\end{eqnarray}

The procedure above implies that the estimates (\ref{eq4.18}) hold. Therefore, we have proved the assertion.

Similar to (\ref{eq4.16}), for $l=1,2, ~0\leq t\leq T$, the Feymann-Kac formula implies that
\begin{eqnarray}\label{eq4.26}
\Phi_l(t,x)=Y_{t,l}^t(x)=\mathbb{E}(m_{0,l}(X_{T,l}^t(x)))+\int_t^T\mathbb{E}J_{v_l}(T-s, X_{s,l}^t(x))ds.
\end{eqnarray}
 Changing the order of expectation and differential, together with Holder's inequality, we have 
\begin{eqnarray}\label{eq4.27}
&&\|\mathbb{E}J_{v_1}(T-s,X_{s,1}^t(\cdot))-\mathbb{E}J_{v_2}(T-s,X_{s,2}^t(\cdot))\|_{W^{k-1,p}}^p\\\nonumber
&\leq& \mathbb{E}\|J_{v_1}(T-s,X_{s,1}^t(\cdot))-J_{v_2}(T-s,X_{s,2}^t(\cdot))\|_{W^{k-1,p}}^p.
\end{eqnarray}
By (\ref{eq4.18}), we deduce that
\begin{eqnarray*}
&&\|J_{v_1}(T-s,X_{s,1}^t(\cdot))-J_{v_2}(T-s,X_{s,2}^t(\cdot))\|_{W^{k-1,p}}^p\\\nonumber
&\leq&Ce^{CKT}(1+TK)^{(k-2)p}\|J_{v_1}(T-s)-J_{v_2}(T-s)\|_{W^{k-1,p}}^p\\\nonumber
&&+CT^pe^{CKT}(1+TK)^{(k-1)p}\sup_{s\in[0,T]}\|J_{v_2}(T-s)\|_{W^{k,p}}^p\sup_{t\in[0,T]}\|v_1(t)-v_2(t)\|_{W^{k-1,p}}^p,~a.s.
\end{eqnarray*}
hence, applying Lemmas \ref{lemma4.2}, \ref{lemma4.4}, 
\begin{eqnarray*}
&&\|J_{v_1}(T-s,X_{s,1}^t(\cdot))-J_{v_2}(T-s,X_{s,2}^t(\cdot))\|_{W^{k-1,p}}^p\\\nonumber
&\leq&C(\alpha)K_0^pe^{CKT}(1+TK)^{kp}\sup_{t\in[0,T]}\|v_1(t)-v_2(t)\|_{W^{k-1,p}}^p\\\nonumber
&&+C(\alpha)K^pe^{CKT}(1+TK)^{(k-2)p}\sup_{t\in[0,T]}\|m_1(t)-m_2(t)\|_{W^{k-1,p}}^p.~~a.s.
\end{eqnarray*}
Putting this into (\ref{eq4.27}), 
\begin{eqnarray*}
&&\|\mathbb{E}J_{v_1}(T-s,X_{s,1}^t(\cdot))-\mathbb{E}J_{v_2}(T-s,X_{s,2}^t(\cdot))\|_{W^{k-1,p}}\\\nonumber
&\leq&C(\alpha)K_0e^{CKT}(1+TK)^{k}\sup_{t\in[0,T]}\|v_1(t)-v_2(t)\|_{W^{k-1,p}}\\\nonumber
&&+C(\alpha)Ke^{CKT}(1+TK)^{k-2}\sup_{t\in[0,T]}\|m_1(t)-m_2(t)\|_{W^{k-1,p}}~~a.s.
\end{eqnarray*}
Similarly, we can get the following estimate
\begin{eqnarray*}
&&\|\mathbb{E}(m_{0,1}(X_{T,1}^t(\cdot)))-\mathbb{E}(m_{0,2}(X_{T,2}^t(\cdot)))\|_{W^{k-1,p}}\\
&\leq&Ce^{CKT}(1+TK)^{k-2}\|m_{0,1}-m_{0,2}\|_{W^{k-1,p}}\\
&&+CTe^{CKT}(1+TK)^{k-1}\|m_{0,2}\|_{W^{k,p}}\sup_{t\in[0,T]}\|v_1(t)-v_2(t)\|_{W^{k-1,p}}.~a.s.
\end{eqnarray*}
Putting the above estimate into (\ref{eq4.26}), conclusion (\ref{eq4.17}) follows.
\end{proof}

\begin{corollary}
Let $v, ~\Phi(t)$ be as in Lemma \ref{lemma4.3}. Then $\Phi\in C([0,T];W^{k,p}(\mathbb{R}^d;\mathbb{R}^d))$ satisfies (\ref{eq4.6}).
\end{corollary}
\begin{proof}
Since $m\in\mathcal{F}(p,p^\prime,T)$, by Theorem 3.2 in \cite{Pardoux-Peng:92}, 
$\Phi(t,x):=Y_{T-t}^{T-t}(x)\in C^1([0,T];C_b^2(\mathbb{R}^d;\mathbb{R}^d))$ is a solution of the  following parabolic partial differential equation (PDE):
\begin{eqnarray*}
\frac{\partial \Phi}{\partial t}+v\cdot \nabla \Phi=\nu\Delta \Phi+J_v, ~\Phi(0)=m_0.
\end{eqnarray*}
Therefore,
\begin{eqnarray*}
\Phi(s,x)-\Phi(t,x)=\int_t^s(\nu \Delta \Phi(r,x)+J_v(r,x)-v(r,x)\cdot \nabla \Phi(r,x))dr.
\end{eqnarray*}
By Lemmas \ref{lemma4.2} and \ref{lemma4.3},
\begin{eqnarray*}
\|\Phi(s)-\Phi(t)\|_{W^{k,p}}\leq C(\alpha)(s-t)\sup_{t\in[0,T]}(\|\Phi(t)\|_{W^{k+2,p}}+\|\Phi(t)\|^2_{W^{2,p}}+\|v(t)\|^2_{W^{k+2,p}}),
\end{eqnarray*}
we obtain $\Phi\in C([0,T];W^{k,p}(\mathbb{R}^d;\mathbb{R}^d))$.
\end{proof}

By Lemmas \ref{lemma4.3}, \ref{lemma4.5}, we can extend the map $\mathcal{P}_\nu$ from $\mathcal{F}(p,p^\prime,T)$ to $\mathcal{B}(m_0,T,p,k)$.
\begin{proposition}\label{prop4.7}
Suppose $T>0,~d<p<\infty,~ k>1$, $m_0\in W^{k,p}(\mathbb{R}^d;\mathbb{R}^d)$ satisfies $\nabla\cdot m_0=0$. Then $\mathcal{P}_{\nu}$ can be extended to be a map $\mathcal{P}_{\nu}:\mathcal{B}(m_0,T,p,k)\rightarrow \mathcal{B}(m_0,T,p,k)$ and for $m_1,m_2\in\mathcal{B}(m_0,T,p,k)$, estimates (\ref{eq4.6}), (\ref{eq4.17}) hold with $\Phi_1,\Phi_2$ replaced by
$\mathcal{P}_{\nu}(m_1),\mathcal{P}_{\nu}(m_2)$.
\end{proposition}
\begin{proof}
Since $C_c^\infty(\mathbb{R}^d;\mathbb{R}^d)$ is dense in $W^{k,p}(\mathbb{R}^d;\mathbb{R}^d)$,  for every $m\in \mathcal{B}(m_0,T,p,k)$, we can find a sequence $\{\tilde{m}_n\}_{n=1}^{\infty}$ such that for every $n,~~ \tilde{m}_n\in C([0,T];C_c^\infty(\mathbb{R}^d;\mathbb{R}^d))$, $\tilde{v}_n:=(I-\alpha^2\Delta)^{-1}\tilde{m}_n$ ($\nabla\cdot\tilde{m}_n=0$ may not hold), and 
\begin{eqnarray*}
\lim_{n\rightarrow \infty}\sup_{t\in[0,T]}\|\tilde{m}_n(t)-m(t)\|_{W^{k,p}}=0.
\end{eqnarray*}
Denote $m_n(t):=\mathbb{P}\tilde{m}_n(t),v_n:=(I-\alpha^2\Delta)^{-1} m_n=\mathbb{P}\tilde{v}_n(t)$, we obtain $m_n\in \mathcal{F}(p,p^\prime,T)$. From \cite{Stein:70} we know that $\mathbb{P}$ is a singular integral operator and bounded in
$W^{k,p}(\mathbb{R}^d)$. Hence
\begin{eqnarray*}
&&\lim_{n\rightarrow\infty}\sup_{t\in[0,T]}\|m_n(t)-m(t)\|_{W^{k,p}}\\
&=&\lim_{n\rightarrow\infty}\sup_{t\in[0,T]}\|\mathbb{P}\tilde{m}_n(t)-\mathbb{P}m(t)\|_{W^{k,p}}\\
&\leq&C\lim_{n\rightarrow\infty}\sup_{t\in[0,T]}\|\tilde{m}_n(t)-m(t)\|_{W^{k,p}}=0,
\end{eqnarray*}
for $m_{0,n}:=m_n(0)$,
\begin{eqnarray}\label{eq4.28}
\lim_{n\rightarrow\infty}\|m_{0,n}-m_0\|_{W^{k,p}}=0.
\end{eqnarray}
Since $\mathbb{P}$ is a singular integral operator, by (\ref{eq4.6}) and (\ref{eq4.17}),
$\{\mathcal{P}_\nu(m_n)\}_{n=1}^{\infty}$ is a Cauchy sequence in $C([0,T];W^{k-1,p}(\mathbb{R}^d;\mathbb{R}^d))$ and there is a $\hat{m}\in C([0,T];W^{k-1,p}(\mathbb{R}^d;\mathbb{R}^d))$
such that,
\begin{eqnarray}\label{eq4.29}
&&\sup_n\sup_{t\in[0,T]}\|\mathcal{P}_\nu(m_n)(t)\|_{W^{k,p}}<\infty, \\
&&\lim_{n\rightarrow\infty}\sup_{t\in[0,T]}\|\mathcal{P}_\nu(m_n)(t)-\hat{m}(t)\|_{W^{k-1,p}}=0.\label{eq4.30}
\end{eqnarray}
From (\ref{eq4.30}) we know that $\nabla\cdot \hat{m}(t)=0$ for every $t$. By definition,
$\mathcal{P}_\nu(m_n)(0)=m_{0,n}$,
then according to (\ref{eq4.28}) and (\ref{eq4.30}), we obtain $\hat{m}(0)=m_0$. Due to (\ref{eq4.17}), the limit $\hat{m}$ we have obtained above is independent of the choice of approximation sequence $\{\tilde{m}_n\}$. Hence $\mathcal{P}_\nu(m):=\hat{m}$ is well defined.
From (\ref{eq4.29}), we know that there exists a subsequence $\{\mathcal{P}_\nu(m_{n_k})\}_{k=1}^{\infty}$ such that 
$$\mathcal{P}_\nu(m_{n_k})\rightarrow \hat{m}~~~\hbox{weakly}$$
and $\hat{m}\in C([0,T];W^{k,p}(\mathbb{R}^d;\mathbb{R}^d))$.
By (\ref{eq4.29}) and (\ref{eq4.30}), (\ref{eq4.17}) holds with $\Phi_1,\Phi_2$ replaced by
$\mathcal{P}_\nu(m_1),\mathcal{P}_\nu(m_2)$ for every $m_1,m_2\in\mathcal{B}(m_0,T,p,k)$. 
\end{proof}

\begin{theorem}\label{thm4.8}
For $d<p<\infty,~k>1$ and $m_0\in W^{k,p}(\mathbb{R}^d;\mathbb{R}^d)$ satisfying that $\nabla\cdot m_0=0$, there exists a constant $T_0$, which depends only on $\|m_0\|_{W^{k,p}}$, such that there is a unique fixed point
 $m$ of the map $\mathcal{P}_\nu$ in $\mathcal{B}(m_0,T_0,p,k)$.
\end{theorem}
\begin{proof}
Suppose that $m \in \mathcal{B}(m_0,T,p,k)$ with $\sup\limits_{t\in[0,T]}\|v(t)\|_{W^{k,p}}\leq K,\sup\limits_{t\in[0,T]}\|m(t)\|_{W^{k,p}}\leq K_0$, by Proposition \ref{prop4.7}, we have
\begin{eqnarray*}
\sup\limits_{t\in [0, T]}\|\mathcal{P}_\nu(m)(t)\|_{W^{k,p}}&\leq& Ce^{CKT}\|m_0\|_{W^{k,p}}(1+TK)^{k-1}+C(\alpha)e^{CKT}TKK_0(1+TK)^{k-1}.
\end{eqnarray*}
Let $T$ tend to $0$; the above bound in the right hand side tends to $C\|m_0\|_{W^{k,p}}$ and $C$ is independent of $K,K_0$. Hence, we can find constants $\tilde{K}_0>>
\|m_0\|_{W^{k,p}}$ and $0<T_1<1$ which only depends on $\|m_0\|_{W^{k,p}}$, such that for every $0<T\leq T_1$, $m\in \mathcal{B}(m_0,T,p,k)$ with 
$\sup\limits_{t\in[0,T]}\|v(t)\|_{W^{k,p}}\leq \tilde{K}_0$, $\sup\limits_{t\in[0,T]}\|m(t)\|_{W^{k,p}}\leq \tilde{K}_0$,
$$\sup\limits_{t\in[0,T]}\|\mathcal{P}_\nu(m)\|_{W^{k,p}}\leq \tilde{K}_0.$$
Fix such $\tilde{K}_0$, by Proposition \ref{prop4.7}, there is a constant $0<T_0\leq T_1$ which only depends on $\|m_0\|_{W^{k,p}}$, such that for each $m_1,m_2
\in\mathcal{B}(m_0,T_0,p,k)$ with $\sup\limits_{t\in[0,T],l=1,2}\|v_l(t)\|_{W^{k,p}}\leq \tilde{K}_0$,\\
$\sup\limits_{t\in[0,T],l=1,2}\|m_l(t)\|_{w^{k,p}}\leq \tilde{K}_0$,
\begin{eqnarray}\label{eq4.31}
&&\sup_{t\in[0,T_0]}\|\mathcal{P}_\nu(m_1)(t)-\mathcal{P}_\nu(m_2)(t)\|_{W^{k-1,p}}\\\nonumber
&\leq&\frac{1}{4}\sup_{t\in[0,T_0]}\|v_1(t)-v_2(t)\|_{W^{k-1,p}}
+\frac{1}{4}\sup_{t\in[0,T_0]}\|m_1(t)-m_2(t)\|_{W^{k-1,p}}\\\nonumber
&\leq&\frac{1}{4}\sup_{t\in[0,T_0]}\|(I-\alpha^2\Delta)^{-1}(m_1(t)-m_2(t))\|_{W^{k-1,p}}
+\frac{1}{4}\sup_{t\in[0,T_0]}\|m_1(t)-m_2(t)\|_{W^{k-1,p}}\\\nonumber
&\leq&\frac{1}{4}\sup_{t\in[0,T_0]}\|m_1(t)-m_2(t)\|_{W^{k-3,p}}
+\frac{1}{4}\sup_{t\in[0,T_0]}\|m_1(t)-m_2(t)\|_{W^{k-1,p}}\\\nonumber
&\leq&\frac{1}{2}\sup_{t\in[0,T_0]}\|m_1(t)-m_2(t)\|_{W^{k-1,p}},
\end{eqnarray}
For every $m\in \mathcal{B}(m_0,T_0,p,k)$, we denote $\|m\|_{W^{k,p},T}:=\sup\limits_{t\in[0,T]}(\|m(t)\|_{W^{k,p}})$. From the discussion above, we know that 
$\mathcal{P}_\nu$ can be viewed as a map $\mathcal{P}_\nu:\mathcal{B}(m_0,T_0,p,k,\tilde{K_0})\rightarrow
\mathcal{B}(m_0,T_0,p,k,\tilde{K_0})$, where
\begin{eqnarray*}
\mathcal{B}(m_0,T_0,p,k,\tilde{K_0}):=\{m\in \mathcal{B}(m_0,T_0,p,k): \|m \|_{W^{k,p},T_0}\leq \tilde{K}_0\}
\end{eqnarray*}
and $\mathcal{P}_\nu$ is contractive with respect to  $\|\cdot\|_{W^{k-1,p},T_0}$ norm.

Following Theorem 2.1 in \cite{Iyer:06}, we choose $m_1\in \mathcal{B}(m_0,T_0,p,k,\tilde{K}_0)$ (for example, $m_1(t):=m_0$ for every 
$t\in[0,T_0]$), and define $m_n:=\mathcal{P}_\nu(m_{n-1})$. By (\ref{eq4.31})
$$\|m_{n+1}-m_{n}\|_{W^{k-1,p},T_0}\leq\frac{1}{2}\|m_{n}-m_{n-1}\|_{W^{k-1,p},T_0},$$
then $\{m_n\}_{n=1}^{\infty}$ has a strong limit $m\in C([0,T_0];W^{k-1,p}(\mathbb{R}^d;\mathbb{R}^d))$ with respect to the norm $\|\cdot\|_{W^{k-1,p},T_0}$. Since $\sup\limits_n\|m_n\|_{W^{k,p},T_0}\leq \tilde{K}_0$, by the same procedure as in the proof of Proposition \ref{prop4.7} and by exercise 9 in page 128 of \cite{Conway:90}, we have $m\in C([0,T_0];W^{k,p}(\mathbb{R}^d;\mathbb{R}^d))$ and $\|m\|_{W^{k,p},T_0}\leq\tilde{K_0}$. By (\ref{eq4.31}),
$$\|\mathcal{P}_\nu(m_n)-\mathcal{P}_\nu(m)\|_{W^{k-1,p},T_0}\leq\frac{1}{2}\|m_{n}-m\|_{W^{k-1,p},T_0}$$
we obtain $\mathcal{P}_\nu(m)=m$. From (\ref{eq4.31}),
we can also obtain the uniqueness of the fixed point.
\end{proof}
\\
\\
\\
\bf Proof of Theorem \ref{thm4.1}. \rm

\noindent\bf Construction of a solution to equation (\ref{eq1.2}).\rm

 Since $m_0\in W^{k,p}(\mathbb{R}^d;\mathbb{R}^d)$, we can find a sequence $\{m_{0,n}\}\subset C_c^\infty(\mathbb{R}^d;\mathbb{R}^d)$, such that $\lim\limits_{n\rightarrow \infty}\|m_{0,n}-m_0\|_{W^{k,p}}=0$ and $\nabla\cdot m_{0,n}=0$. Using the iteration procedure in the proof of Theorem \ref{thm4.8}, we can find a constant $T_1$ independent of $n$,
such that for every $n$, there exist vectors $\{m_{n,l}\}_{l=1}^{\infty}\subset C([0,T_1]; C_c^\infty(\mathbb{R}^d;\mathbb{R}^d))$, $m_n\in C([0,T_1]; W^{k,p}(\mathbb{R}^d;\mathbb{R}^d))$, 
such that 
\begin{eqnarray}\label{eq4.32}
&&m_{n,l}(0)=m_{0,n}, ~~m_{n,l+1}=\mathcal{P}_{\nu}(m_{n,l}),~~\sup_{n,l}\sup_{t\in[0,T_1]}\|m_{n,l}(t)\|_{W^{k,p}}<\infty,\nonumber\\
&&\lim_{l\rightarrow\infty}\sup_{t\in[0,T_1]}\|m_{n,l}(t)-m_n(t)\|_{W^{k-1,p}}=0,
~~\mathcal{P}_\nu(m_n)=m_n.
\end{eqnarray}
Let $(X_{n,l},Y_{n,l},Z_{n,l})$ be the solution of (\ref{eq4.3}) with coefficients $v=u_{n,l}, u_{n,l}
=(I-\alpha^2\Delta)^{-1}m_{n,l}$ and initial condition $m_{n,l}(0)=m_{0,n}$. We denote $\Phi_{n,l}(t):=Y_{T_1-t,n,l}^{T_1-t}$ for $t\in[0,T_1]$. Since $u_{n,l}$ is regular enough, $\Phi_{n,l}$ is the unique classical solution of the following PDE,
\begin{eqnarray}\label{eq4.33}
\partial_t\Phi_{n,l}+u_{n,l}\cdot\nabla \Phi_{n,l}=\nu\Delta \Phi_{n,l}+J_{u_{n,l}},~~\Phi_{n,l}(0)=m_{0,n}.
\end{eqnarray}
Hence, it is a strong solution in the following sense, for every $t\in [0,T_1]$,
\begin{eqnarray}\label{eq4.34}
\Phi_{n,l}(t)=e^{t\nu\Delta}m_{0,n}-\int_0^te^{(t-s)\nu\Delta}(u_{n,l}(s)\cdot\nabla \Phi_{n,l}(s)-J_{u_{n,l}}(s))ds.
\end{eqnarray}
For $T_1$ independent of $n$, $l$ small enough, by (\ref{eq4.17}) we obtain,
\begin{eqnarray*}
&&\sup_{t\in[0,T_1]}\|\Phi_{n,l+1}(t)-\Phi_{n,l}(t)\|_{W^{k-1,p}}\\
&\leq& 2\sup_{t\in[0,T_1]}\|m_{n,l+1}-m_{n,l}\|_{W^{k-1,p}}+2\sup_{t\in[0,T_1]}\|u_{n,l+1}-u_{n,l}\|_{W^{k-1,p}}\\
&\leq& 2\sup_{t\in[0,T_1]}\|m_{n,l+1}-m_{n,l}\|_{W^{k-1,p}}+2\sup_{t\in[0,T_1]}\|m_{n,l+1}-m_{n,l}\|_{W^{k-3,p}}.
\end{eqnarray*}
By (\ref{eq4.32}) and the same argument in the proof of Proposition \ref{prop4.7}, there is a $\Phi_n\in C([0,T_1];W^{k,p}(\mathbb{R}^d;\mathbb{R}^d))$ such that
\begin{eqnarray*}
\sup_n\sup_{t\in[0,T_1]}\|\Phi_n(t)\|_{W^{k,p}}<\infty,~~
\lim_{n\rightarrow\infty}\sup_{t\in[0,T_1]}\|\Phi_{n,l}(t)-\Phi_n(t)\|_{W^{k-1,p}}=0.
\end{eqnarray*}
Let $l\rightarrow \infty$ in (\ref{eq4.34}); we get, for every $t\in [0,T_1]$,
\begin{eqnarray}\label{eq4.35}
\Phi_{n}(t)=e^{t\nu\Delta}m_{0,n}-\int_0^te^{(t-s)\nu\Delta}(u_{n}(s)\cdot\nabla \Phi_{n}(s)-J_{u_{n}}(s))ds.
\end{eqnarray}

Since $\nabla\cdot m_{n,l}(t)=0$, by definition of $F_{u_{n,l}}$ and$J_{u_{n,l}}$, we have
\begin{eqnarray*}
&&\nabla \cdot F_{u_{n,l}}(t)=\nabla\cdot(\nabla NG_{u_{n,l}}(t))=\Delta NG_{u_{n,l}}(t)=G_{u_{n,l}}(t)\\
&&\nabla \cdot J_{u_{n,l}}(t)=\nabla \cdot F_{u_{n,l}}(t)+\alpha^2\nabla \cdot (\nabla u_{n,l}^T\Delta u_{n,l})=\sum_{i,j=1}^d \partial_iu_{n,l}^j(t)\partial_jm_{n,l}^i(t).
\end{eqnarray*}
Denote
\begin{eqnarray*}
H_{u_{n,l},\Phi_{n,l}}(t):=\sum_{i,j=1}^d \partial_iu_{n,l}^j(t)\partial_j(\Phi_{n,l}^i(t)-m_{n,l}^i(t)).
\end{eqnarray*}
Let $h_{n,l}(t):=\nabla\cdot \Phi_{n,l}(t)$; taking the divergence in (\ref{eq4.33}), we obtain, for every $t\in [0,T_1]$,
\begin{eqnarray*}
\partial_th_{n,l}+u_{n,l}\cdot\nabla h_{n,l}=\nu\Delta h_{n,l}-H_{u_{n,l},\Phi_{n,l}},~~h_{n,l}(0)=0.
\end{eqnarray*}
Applying It\^o's formula to $h_{n,l}(T_1-s,X^t_{s,n,l}(x))$ and taking the expectation, we obtain
\begin{eqnarray*}
h_{n,l}(T_1-t,x)=-\int_t^{T_1}\mathbb{E}(H_{u_{n,l},\Phi_{n,l}}(T_1-s,X_{s,n,l}^t))ds,
\end{eqnarray*}
so we have
\begin{eqnarray*}
\|h_{n,l}(t)\|_{L^p}^p\leq C\int_0^{t}\|H_{u_{n,l},\Phi_{n,l}}(s)\|_{L^p}^pds.
\end{eqnarray*}
By (\ref{eq4.32}), let $l\rightarrow \infty$, 
\begin{eqnarray}\label{eq4.36}
\|h_{n}(t)\|_{L^p}^p\leq C\int_0^{t}\|H_{u_{n},\Phi_{n}}(s)\|_{L^p}^pds,
\end{eqnarray}
where $h_n(t):=\nabla\cdot \Phi_n(t)$ and 
\begin{eqnarray*}
H_{u_{n},\Phi_{n}}(t):=\sum_{i,j=1}^d \partial_iu_{n}^j(t)\partial_j(\Phi_{n}^i(t)-m_{n}^i(t)).
\end{eqnarray*}
Since for every $m \in C_c^\infty(\mathbb{R}^d;\mathbb{R}^d)$, the Leray-Hodge projection has the expression $m-\mathbb{P}m =\nabla N(\nabla\cdot m)$, we obtain, for every $p>1$,
\begin{eqnarray*}
\|\nabla(m-\mathbb{P} m)\|_{L^p}=\|\nabla^2N(\nabla\cdot m)\|_{L^p}
\leq C\|\nabla\cdot m\|_{L^p}.
\end{eqnarray*}
Since $\mathbb{P}(\Phi_{n,l}(t))=m_{n,l+1}(t)$, we have $\mathbb{P}(\Phi_n(t))=m_n(t)$. Therefore,
\begin{eqnarray*}
\|\nabla(m_n(t)-\Phi_n(t))\|_{L^p}\leq C\|\nabla\cdot \Phi_n(t)\|_{L^p}=C\|h_n(t)\|_{L^p},
\end{eqnarray*}
which yields that 
\begin{eqnarray}\label{eq4.37}
\|H_{u_{n},\Phi_{n}}\|_{L^p}\leq CK\|h_n(t)\|_{L^p}
\end{eqnarray}
where $K:=\sup\limits_{n}\sup\limits_{t\in[0,T]}(\|\nabla u_n(t)\|_{L^\infty}).$ By (\ref{eq4.36}), (\ref{eq4.37}) and Gr\"onwall's inequality, we deduce that $\|h_n(t)\|_{L^p}=0$ for every $t\in [0,T_1]$. Hence, $\nabla \cdot \Phi_n(t)=0$
and $\Phi_n(t)=\mathbb{P}\Phi_n(t)=m_n(t)$.

Since $\mathcal{P}_{\nu}(m_n)=m_n$, by (\ref{eq4.17}), there is a $0<T_0\leq T_1$ 
independent of $n$ and a vector $m\in C([0,T_0];W^{k,p}(\mathbb{R}^d;\mathbb{R}^d))$,
such that 
\begin{eqnarray*}
\lim_{n\rightarrow \infty} \sup_{t\in[0,T_0]}\|m_n(t)-m(t)\|_{W^{k-1,p}}=0;
\end{eqnarray*}
then let $n\rightarrow \infty$ in (\ref{eq4.35}) we obtain, for every $t\in [0,T_0]$,
\begin{eqnarray*}
m(t)=e^{t\nu\Delta}m_{0}-\int_0^te^{(t-s)\nu\Delta}(u(s)\cdot\nabla m(s)-J_{u}(s))ds.
\end{eqnarray*}
Therefore, $m\in C([0,T_0];W^{k,p}(\mathbb{R}^d;\mathbb{R}^d))$ is the strong solution of (\ref{eq1.2}).
\\
\\
\noindent\bf Uniqueness.\rm

Suppose $m\in C([0,T_0];W^{k,p}(\mathbb{R}^d;\mathbb{R}^d))$ is a strong solution of (\ref{eq1.2})
and $T_0$ is small enough. 
Under such regularity condition, the FBSDE (\ref{eq4.3}) with coefficient 
$u=(I-\alpha^2\Delta)^{-1}m$ and initial condition $m(0,x)=m_0(x)$ has a unique solution $(X,Y,Z)$.
Denote  $\Phi(t):=Y_{T_0-t}^{T_0-t}$ for $t\in [0,T_0]$. By (\ref{eq4.17}) and the approximation procedure above,  $\Phi\in C([0,T_0];W^{k,p}(\mathbb{R}^d;\mathbb{R}^d))$ is the strong solution of the following 
(linear) PDE,
\begin{eqnarray}\label{eq4.38}
\partial_t\Phi+u\cdot\nabla \Phi=\nu\Delta \Phi+J_{u},~~\Phi(0)=m_{0}.
\end{eqnarray}
On the other hand, since $m$ is a strong solution of (\ref{eq1.2}), $m$ is also a strong solution of (\ref{eq4.38}). By the uniqueness of the strong solution of the linear PDE (\ref{eq4.38}) in such function space, we must have $\Phi(t)=m(t)$, so $m=\Phi=\mathcal{P}_\nu(m)$, hence it is a fixed point of $\mathcal {P}_\nu$ in 
$\mathcal{B}(m_0,T_0,p,k)$ and by Theorem \ref{thm4.8} it is unique.

\begin{remark}
From the proceeding of our proof, we can directly show that Theorem \ref{thm4.1} holds for the d-dimensional $(d\geq 3)$ Leray-$\alpha$ equation, that is, equation (\ref{eq1.2}) without the term 
$\alpha^2\nabla u^T\Delta u$.
\end{remark}

\bf Acknowledgements: \rm  
This research was supported by China Scholarship Council.

\providecommand{\bysame}{\leavevmode\hbox to3em{\hrulefill}\thinspace}

\end{document}